\documentclass[pdflatex,sn-mathphys-num]{sn-jnl}

\makeatletter
\def\pacsbullet{\hbox{,\hskip\fontdimen2\font}}
\makeatother

\usepackage[utf8]{inputenc}
\usepackage[T1]{fontenc}
\usepackage{graphicx}%
\usepackage{multirow}%
\usepackage{amsmath,amssymb,amsfonts}%
\usepackage{amsthm}%
\usepackage{mathrsfs}%
\usepackage[title]{appendix}%
\usepackage{xcolor}%
\usepackage{textcomp}%
\usepackage{manyfoot}%
\usepackage{booktabs}%
\usepackage{algorithm}%
\usepackage{algorithmicx}%
\usepackage{algpseudocode}%
\usepackage{listings}%
\usepackage{cleveref}
\usepackage{enumitem}

\DeclareMathOperator{\Tr}{tr}
\newcommand{\R}{\mathbb{R}}
\newcommand{\E}{\mathbb{E}}
\newcommand{\norm}[2]{\|#1\|_{#2}}
\newcommand{\Anorm}[1]{\|#1\|_{A}}
\newcommand{\lmin}{\lambda_{\min}}
\newcommand{\lmax}{\lambda_{\max}}

\theoremstyle{thmstyleone}%
\newtheorem{theorem}{Theorem}%
\newtheorem{proposition}[theorem]{Proposition}%

\newtheorem{corollary}[theorem]{Corollary}
\newtheorem{lemma}[theorem]{Lemma}

\theoremstyle{thmstyletwo}%
\newtheorem{remark}{Remark}%

\theoremstyle{thmstylethree}%
\newtheorem{definition}{Definition}%
\newtheorem{assumption}{Assumption}%

\begin{document}
\title[Residual-Weighted Randomized Jacobi]{Residual-Weighted Randomized Jacobi: Sharpened Bounds via Residual Concentration and Asynchronous Extension}

\author*[1]{\fnm{Evan} \sur{Coleman}}\email{ecolema4@umw.edu}

\affil*[1]{\orgdiv{Department of Computer Science}, \orgname{University of Mary Washington}, \city{Fredericksburg},  \state{Virginia}, \country{USA}\footnote{ORCID: 0000-0002-3399-1117}}

\abstract{
We study randomized stationary methods for symmetric positive definite
linear systems in which component $j$ is selected with probability
proportional to $|r_j|^\ell$. This family interpolates between uniform
sampling on the nonzero residual components ($\ell \to 0^+$) and
Gauss--Southwell relaxation ($\ell \to \infty$). For $\ell = 2$ we sharpen the one-step
analysis using the inverse participation ratio (IPR) of the
residual, $\nu^2(r) = n\|r\|_4^4/\|r\|_2^4$, which is
$1$ for a uniform residual and approaches $n$ as it
concentrates. The resulting bound multiplies the guaranteed
expected-progress coefficient of uniform sampling by the IPR (and the
expected progress, for constant-diagonal $A$).

We extend the analysis to asynchronous execution via the Avron--Druinsky--Gupta (ADG) framework, obtaining, under an a priori envelope on the IPR, an epoch-based reduction bound in which the IPR controls both the progress coefficient and
the certified delay window. Synchronous runs on three structured problems
and eight Matrix Market matrices support the sharpened bound, and
shared-memory experiments show an IPR trajectory nearly
insensitive to concurrency. There, committing a buffered residual value
(consistent reads, the easier case for ADG) led to divergence at high
concurrency on every problem, while re-reading the component at
update time (inconsistent reads) diverged in a minority of runs on one
problem. We conjecture a feedback-damping mechanism, test its predictions,
and find that sufficiently small fixed steps and an adaptive IPR-driven
step-size rule removed the instability in every run, at a cost in speed.}

\keywords{Asynchronous iterative methods, Randomized Jacobi, Inverse participation ratio, Residual-weighted sampling}

\pacs[MSC Classification]{65F10, 65Y05, 68W20, 68W10}

\maketitle

\section{Introduction}
\label{sec:intro}

Randomized iterative methods for solving linear systems have become a standard
tool in numerical linear algebra, optimization, and scientific computing.  In
row-action and coordinate-relaxation methods, each iteration updates only one
row, coordinate, component, or block of the current approximation.  Randomized
Kaczmarz methods \cite{strohmer2009randomized}, randomized coordinate descent,
and randomized Gauss--Seidel-type schemes \cite{leventhal2010randomized}
illustrate the appeal of this approach: each step is inexpensive, the update
formula is simple, and the method can often be implemented with relatively low
synchronization overhead.  These properties make randomized stationary methods
especially natural candidates for asynchronous parallel computation
\cite{avron2015revisiting}.

A central design choice in such methods is the component-selection rule.  The
simplest choice is uniform random sampling, which is easy to implement and
analyze.  More refined fixed distributions, such as row-norm or
diagonal-weighted sampling, incorporate matrix geometry but remain independent
of the current residual.  At the opposite extreme are greedy rules such as
Southwell relaxation \cite{southwell1946relaxation}, which
selects the component with the largest residual or largest predicted progress.
Greedy selection can be substantially more efficient per iteration, but it
requires global residual comparisons, sorting, or priority maintenance.  These
operations can dominate the cost of a relaxation step and are particularly
awkward in asynchronous or distributed settings.

This paper studies a middle ground between fixed randomized sampling and fully
greedy relaxation.  Given the current residual $r^{(i)}=b-Ax^{(i)}$, we select
component $j$ with probability
\[
    P(k=j) =
    \frac{|r_j^{(i)}|^\ell}{\sum_m |r_m^{(i)}|^\ell},
    \qquad \ell>0.
\]
This \emph{power-weighted} family is residual-adaptive but does not require
sorting, thresholding, or constructing an eligible set.  It also has transparent
limiting behavior: as $\ell \to 0$ the method approaches uniform randomized Jacobi, while as
$\ell \to \infty$ it concentrates on the largest-magnitude components and approaches 
Gauss--Southwell selection.  The
parameter $\ell$ therefore controls how aggressively the method concentrates
sampling effort on large-residual components.

The motivation for this family is simple.  Uniform sampling treats all
components as equally important, even when most of the residual mass is
concentrated in only a few entries.  In such a regime, many uniform updates are
spent on components that can contribute little immediate progress.  Residual-weighted sampling instead increases the probability of selecting components
that currently dominate the residual.  The key question is whether this
intuition can be made quantitative.  In particular, one would like a convergence
bound that reduces to the uniform bound when the residual is flat, but becomes
strictly sharper when the residual is non-uniform.

For synchronous randomized Jacobi, we answer this question using the
inverse participation ratio (IPR) of the residual,
\[
\nu^2(r) = \frac{n \cdot \norm{r}{4}^4}{\norm{r}{2}^4},
\]
a scalar that equals one when all residual components have equal magnitude
and grows toward $n$ as the residual concentrates on a single component.
For power-weighted sampling with $\ell = 2$, the expected one-step progress
naturally contains a residual-norm ratio that rewrites as a multiple of the
IPR, and the resulting sharpened bound multiplies the guaranteed per-step
progress coefficient by exactly that scalar. The same
quantity that explains \emph{when} adaptive sampling helps therefore
quantifies \emph{how much} it helps.

The asynchronous setting is more delicate. In uniform asynchronous Jacobi
the component choice is independent of the stale iterate, and the
interference between simultaneously updated components is averaged away by
the uniform sampling distribution. Residual-weighted sampling breaks both
features: the sampling probabilities depend on the stale residual seen by
the worker, coupling the choice to the delay, and concentrated residuals
produce concentrated sampling distributions whose interference is no longer
smoothed by that averaging. Residual concentration therefore increases
expected progress but can also increase the worst-case interference bound,
and the same scalar controls both. We extend the Avron--Druinsky--Gupta
(ADG) framework for uniform asynchronous Jacobi~\cite{avron2015revisiting}
to this residual-dependent setting under a global IPR envelope, obtaining 
an epoch-based reduction bound in which the IPR
controls both the progress coefficient and the certified delay window.

Synchronous runs and shared-memory experiments on a 96-core node
accompany the two halves of the analysis. On the three test 
problems the IPR sits well above its floor of
$1$, with sustained values between roughly $3.5$ and $6$ and much larger
transient values under localized forcing, and its trajectory is essentially
insensitive to thread count. The least expected finding concerns the read
model: committing a buffered residual value (ADG's consistent reads, the
case the theory covers) led to divergence at high concurrency on every
test problem, while re-reading the component at update time diverged in
a minority of runs on one problem and in none on the others; the same
IPR that amplifies progress also amplifies the thread-collision rate, and
the refreshed read appears to absorb it.

Propositions~\ref{thm:leventhal-lewis} and~\ref{thm:uniform} restate
classical sampling bounds, and Proposition~\ref{thm:avron-async} recalls
the ADG result~\cite{avron2015revisiting}. The synchronous sharpened
bounds of Section~\ref{sect:baseline-residual-weighted} are
coordinate-descent counterparts of results of Gower et
al.~\cite{gower2021adaptive} and
Steinerberger~\cite{steinerberger2021weighted}; Section~\ref{sect:sampling-adaptivity-spectrum}
gives the correspondence. What is new there is the identification of the
sharpening factor with the IPR and the properties, exactness and
convergence statements that follow from it; the asynchronous theory and
the experiments are new.

The contributions, in order of appearance, are as follows. The paper
\begin{enumerate}
    \item organizes randomized, adaptive and greedy relaxation rules into a
    component-selection adaptivity spectrum that places residual-weighted
    sampling between fixed randomized methods and Southwell-type greedy
    relaxation (Section~\ref{sect:sampling-adaptivity-spectrum});

    \item proves, for every $\ell > 0$, a worst-case bound matching uniform 
    randomized Jacobi, hence convergence in expectation, and, for $\ell = 2$, a 
    sharpened one-step bound whose guaranteed progress coefficient is multiplied 
    by the IPR of the residual (Section~\ref{sect:baseline-residual-weighted});

    \item extends the analysis to asynchronous execution for $\ell = 2$,
    obtaining an ADG-style epoch reduction bound under a global IPR envelope, 
    with a step-size prescription for delays outside
    the full-step window (Section~\ref{sec:async});

    \item tests the sharpened bound on structured problems and Matrix
    Market matrices, measures the IPR trajectory under emergent asynchrony, and
    documents the consistent-reads instability together with a conjectured
    mechanism, tests of its predictions about problem size, thread count
    and exponent, and the step-size rule that removed the instability in
    every run (Section~\ref{sec:experiments}).
\end{enumerate}

The rest of the paper is organized as follows. \Cref{sec:related-work}
reviews related work, and \Cref{sec:background} introduces the notation,
the one-step progress identity and the sampling baselines.
\Cref{sect:baseline-residual-weighted} and \Cref{sec:async} develop the
synchronous and asynchronous analyses. \Cref{sec:experiments} presents
the numerical experiments, and \Cref{sec:conclusion} concludes.

\section{Related Work}
\label{sec:related-work}

Asynchronous iterative methods have a long history in numerical linear algebra,
beginning with the chaotic relaxation framework of Chazan and Miranker
\cite{chazan1969chaotic} and the multiprocessor analysis of Baudet
\cite{baudet1978asynchronous}.  These early works established that classical
stationary iterations can converge even when processors update components using
stale information, provided the delays and communication patterns satisfy
appropriate regularity assumptions.  The broader theory of asynchronous parallel
computation was later developed systematically by Bertsekas and Tsitsiklis
\cite{bertsekas1989parallel}, and the mathematical foundations and convergence
conditions for asynchronous iterations were surveyed by Frommer and Szyld
\cite{frommer2000asynchronous}.  In this classical literature, the primary goal
is to characterize when asynchrony preserves convergence; sharp convergence
rates and adaptive component-selection rules are typically secondary concerns.

A probabilistic viewpoint on asynchronous iteration was developed by Strikwerda
\cite{strikwerda2002probabilistic}, who extended deterministic asynchronous
models by allowing both component choices and delays to be described
probabilistically.  This perspective is particularly relevant for modern
parallel and distributed implementations, where update order, communication
delays, and processor availability may be irregular or random.  More recently,
Avron, Druinsky, and Gupta revisited asynchronous linear solvers for symmetric
positive definite systems and provided convergence-rate bounds for randomized
asynchronous Jacobi under bounded-delay and independence assumptions
\cite{avron2015revisiting}.  Their work is the closest theoretical point of
comparison for the asynchronous analysis in this paper: like their model, we
study randomized component updates applied to SPD systems, but we replace fixed
uniform sampling with residual-dependent sampling and therefore must account for
the coupling between stale residuals, sampling probabilities, and delayed
updates.  Carson and Ma~\cite{carson2026asynchronous} recently gave a unified
convergence-rate analysis of asynchronous Jacobi and randomized Gauss--Seidel
methods in shared and distributed memory, in which an update may be computed
from values of different past iterations; for uniformly sampled directions
they prove linear convergence in expectation, with a bound that depends on
the delay algebraically.

Asynchronous stationary iterations are also of practical interest because they
can reduce synchronization overhead in settings where exact high-accuracy solves
are not required at every stage.  Chow and Patel developed fine-grained parallel
algorithms for incomplete LU factorization, motivated by the use of incomplete
factorizations as scalable preconditioners for linear solvers
\cite{chow2015fine}.  Related work has studied asynchronous Jacobi iterations
both as standalone solvers and as building blocks for larger methods.  For
example, Wolfson-Pou and Chow developed convergence models for asynchronous
Jacobi and observed behavior that can differ substantially from simple delay
models \cite{wolfsonpou2018convergence}.  They also investigated asynchronous
multigrid methods, where asynchronous relaxation can act as a smoother and help
reduce synchronization costs within multilevel solvers
\cite{wolfsonpou2019multigrid}.  Chow, Frommer, and Szyld analyzed
asynchronous Richardson iterations, further clarifying the relationship between
classical stationary methods, relaxation parameters, and asynchronous execution
\cite{chow2021asynchronous}.  More recent work has extended these ideas to
asynchronous semi-iterative and Chebyshev-type methods, including multigrid
preconditioning \cite{wolfsonpou2025asynchronous}.

Application-oriented studies have shown asynchronous linear iterations to be
useful in large-scale PDE contexts~\cite{glusa2020scalable,kashi2021asynchronous,kashi2025effectiveness},
particularly as low-synchronization components within preconditioners,
smoothers, and domain-decomposition solvers rather than as replacements for
Krylov methods in high-accuracy regimes.
The same synchronization--staleness tradeoff appears in neighboring areas:
asynchronous coordinate-update methods such as HOGWILD!~\cite{recht2011hogwild},
asynchronous stochastic coordinate descent~\cite{liu2015asynchronous}, and the
ARock fixed-point framework~\cite{peng2016arock} all exploit it in optimization
and machine learning, providing useful context even though the present paper
focuses on stationary solvers for linear systems.
For row-action methods, Liu, Wright, and Sridhar analyzed an asynchronous
parallel randomized Kaczmarz algorithm with fixed sampling
probabilities~\cite{liu2014asynchronous}, and residual-adaptive row
selection under shared-memory asynchrony has been studied
experimentally~\cite{coleman2026adaptive}.
Finally, asynchronous methods are closely connected to questions of resilience
and irregular execution.  In distributed or cloud-based environments, workers
may be delayed, temporarily unavailable, or subject to heterogeneous performance.
Recent straggler-tolerant stationary methods exploit this observation by
allowing iterations to proceed using partial or delayed information
\cite{kalantzis2025straggler}.  Fault-tolerant variants of asynchronous linear
solvers have also been studied in the context of data corruption and recovery
\cite{coleman2021fault,coleman2026fault,vogl2024modifying}.  These directions are not the focus
of the present paper, but they reinforce the broader motivation for analyzing
stationary methods under stale and irregular update patterns.

The present work examines how residual-dependent component selection
interacts with asynchronous stationary iteration: whether the selection
rule itself can be biased toward high-residual entries while retaining
useful convergence guarantees, and how large the resulting advantage is.

\section{Background}
\label{sec:background}

\subsection{Setup and Notation}

Let $A \in \R^{n \times n}$ be symmetric positive definite (SPD) and let
$b \in \R^n$.  We consider the linear system
\[
    Ax = b,
\]
whose unique solution is denoted by $x^* = A^{-1}b$.  Since $A$ is SPD, it has
real positive eigenvalues
\[
    0 < \lambda_1 \leq \lambda_2 \leq \cdots \leq \lambda_n,
\]
and we write
\[
    \lmin(A) = \lambda_1, \qquad
    \lmax(A) = \lambda_n, \qquad
    \kappa = \frac{\lmax(A)}{\lmin(A)}.
\]
When the matrix is clear from context, we write simply $\lmin$ and $\lmax$.

At iteration $i$, the current approximate solution is $x^{(i)}$.  We denote the
error and residual by
\[
    e^{(i)} = x^{(i)} - x^*, \qquad
    r^{(i)} = b - A x^{(i)}.
\]
Since $Ax^*=b$, the residual and error are related by
\[
    r^{(i)} = -A e^{(i)}.
\]
We measure error in the energy norm induced by $A$,
\[
    \Anorm{z}^2 = (z,z)_A=z^T A z.
\]
We also use the standard basis vectors $e_1,\ldots,e_n$, where $e_j$ denotes
the vector with a $1$ in component $j$ and zeros elsewhere.  To avoid confusing
the basis vector $e_j$ with the iteration error $e^{(i)}$, the iteration index
will always be written in parentheses.  The diagonal entries of $A$ are denoted
by $A_{jj}$, and we define
\[
    A_{\max} = \max_{1 \leq j \leq n} A_{jj}.
\]
For SPD matrices, $A_{jj} = e_j^T A e_j > 0$ for every $j$, so all coordinate
relaxations below are well-defined. We keep the diagonal entries explicit
throughout the synchronous analysis because the baseline sampling
distributions depend on them, even though symmetric diagonal scaling would give a unit-diagonal system
(residual-weighted probabilities then refer to the scaled residual, a
distinction Section~\ref{sec:experiments} returns to); this keeps residual-weighted sampling directly comparable to both uniform and
diagonal-weighted sampling.

We will repeatedly use the spectral inequalities
\[
    \lmin \norm{z}{2}^2
    \leq
    \Anorm{z}^2
    \leq
    \lmax \norm{z}{2}^2,
\]
and, since $r=-Ae$,
\begin{equation}
\label{eq:residual-energy-bounds}
    \lmin \Anorm{e}^2
    \leq
    \norm{r}{2}^2
    \leq
    \lmax \Anorm{e}^2.
\end{equation}
The lower bound in~\eqref{eq:residual-energy-bounds} is the main link
between the residual-based progress term and convergence in the $A$-norm.

\subsection{Randomized Jacobi Relaxation}

A single Jacobi relaxation updates one component of $x$ using the current
residual; throughout, one iteration is one such coordinate relaxation and
a \emph{sweep} is $n$ consecutive iterations.  If component $k$ is selected at iteration $i$, the update is
\begin{equation}
\label{eq:update}
    x^{(i+1)}
    =
    x^{(i)} + \frac{r_k^{(i)}}{A_{kk}}\, e_k .
\end{equation}
The effect of one coordinate relaxation on the $A$-norm error is exact.  Since
$r_k^{(i)} = -e_k^T A e^{(i)}$, the new error is
\[
    e^{(i+1)}
    =
    e^{(i)} + \frac{r_k^{(i)}}{A_{kk}} e_k .
\]
Expanding the $A$-norm gives
\begin{align}
    \Anorm{e^{(i+1)}}^2
    &=
    \left(e^{(i)} + \frac{r_k^{(i)}}{A_{kk}}e_k\right)^T
    A
    \left(e^{(i)} + \frac{r_k^{(i)}}{A_{kk}}e_k\right) \notag \\
    &=
    \Anorm{e^{(i)}}^2
    + 2\frac{r_k^{(i)}}{A_{kk}} e_k^T A e^{(i)}
    + \left(\frac{r_k^{(i)}}{A_{kk}}\right)^2 A_{kk} \notag \\
    &=
    \Anorm{e^{(i)}}^2
    - 2\frac{(r_k^{(i)})^2}{A_{kk}}
    + \frac{(r_k^{(i)})^2}{A_{kk}}.
\end{align}
Therefore
\begin{equation}
\label{eq:progress-identity}
    \Anorm{e^{(i+1)}}^2
    =
    \Anorm{e^{(i)}}^2
    -
    \frac{(r_k^{(i)})^2}{A_{kk}} .
\end{equation}
This identity is deterministic and does not depend on how the component $k$ is
chosen.  Every exact coordinate relaxation weakly decreases the $A$-norm error,
and the amount of progress is precisely the squared residual in the selected
component, scaled by the corresponding diagonal entry.

Randomization enters only through the choice of $k$.  Let
\[
    p_j^{(i)} = P(k=j \mid x^{(i)}), \qquad j=1,\ldots,n,
\]
where the distribution may be fixed in advance or may depend on the current
iterate through the residual $r^{(i)}$.  
Taking conditional expectation in
\eqref{eq:progress-identity} gives the basic expected-progress formula
\begin{equation}
\label{eq:expected-progress-general}
    \E\!\left[\Anorm{e^{(i+1)}}^2 \mid x^{(i)}\right]
    =
    \Anorm{e^{(i)}}^2
    -
    \sum_{j=1}^n p_j^{(i)}
    \frac{(r_j^{(i)})^2}{A_{jj}} .
\end{equation}
Thus the analysis of any randomized component-selection rule reduces to lower
bounding the expected progress term
\[
    \sum_{j=1}^n p_j^{(i)}
    \frac{(r_j^{(i)})^2}{A_{jj}}.
\]
Uniform sampling, diagonal-weighted sampling, greedy sampling, and the
residual-weighted methods studied later differ only in how they choose the
probabilities $p_j^{(i)}$.
Note that if $r^{(i)}=0$, then $x^{(i)}=x^*$ and the method has already converged.  In
that case the sampling distribution is irrelevant.  Throughout the analysis of
residual-dependent probabilities, we implicitly restrict attention to iterations
with $r^{(i)} \neq 0$.

We recall two classical fixed-distribution baselines, restated with their
short proofs in the present notation because they fix the proof template
that the residual-weighted analysis follows: the diagonal-weighted rate of
Leventhal and Lewis~\cite{leventhal2010randomized} for randomized
coordinate descent on SPD systems, and the uniform-sampling bound (see,
e.g.,~\cite{avron2015revisiting} for the unit-diagonal case) that is the
natural baseline for asynchronous Jacobi.

\begin{proposition}[Diagonal-weighted sampling; Leventhal--Lewis~\cite{leventhal2010randomized}]
\label{thm:leventhal-lewis}
Suppose that at each iteration,
\[
    P(k=j) = \frac{A_{jj}}{\Tr(A)}.
\]
Then
\[
    \E\!\left[\Anorm{e^{(i+1)}}^2 \mid x^{(i)}\right]
    \leq
    \left(1 - \frac{\lmin(A)}{\Tr(A)}\right)
    \Anorm{e^{(i)}}^2.
\]
\end{proposition}

\begin{proof}
The expected progress term in~\eqref{eq:expected-progress-general} equals
\[
    \sum_{j=1}^n \frac{A_{jj}}{\Tr(A)}\,\frac{(r_j^{(i)})^2}{A_{jj}}
    = \frac{\norm{r^{(i)}}{2}^2}{\Tr(A)}
    \geq \frac{\lmin(A)}{\Tr(A)}\,\Anorm{e^{(i)}}^2
\]
by the lower bound in~\eqref{eq:residual-energy-bounds}, and the claim
follows from~\eqref{eq:expected-progress-general}.
\end{proof}

\begin{proposition}[Uniform sampling]
\label{thm:uniform}
Suppose that at each iteration,
\[
    P(k=j) = \frac{1}{n}.
\]
Then
\[
    \E\!\left[\Anorm{e^{(i+1)}}^2 \mid x^{(i)}\right]
    \leq
    \left(1 - \frac{\lmin(A)}{n A_{\max}}\right)
    \Anorm{e^{(i)}}^2.
\]
\end{proposition}

\begin{proof}
The expected progress term in~\eqref{eq:expected-progress-general} equals
\[
    \frac{1}{n}\sum_{j=1}^n \frac{(r_j^{(i)})^2}{A_{jj}}
    \geq \frac{\norm{r^{(i)}}{2}^2}{n A_{\max}}
    \geq \frac{\lmin(A)}{n A_{\max}}\,\Anorm{e^{(i)}}^2,
\]
using $A_{jj} \leq A_{\max}$ and the lower bound
in~\eqref{eq:residual-energy-bounds}.
\end{proof}

Note that since $\Tr(A) = \sum_{j=1}^n A_{jj} \leq n A_{\max}$,
the diagonal-weighted bound is at least as strong as the crude uniform bound.
When $A$ has constant diagonal, the two sampling distributions coincide and
$\Tr(A)=nA_{\max}$, so the two bounds are identical.  This is the case for many
model discretizations, including standard finite-difference Laplacians with
constant diagonal.

Both proofs apply~\eqref{eq:expected-progress-general} and lower bound the
expected progress term by a multiple of $\norm{r^{(i)}}{2}^2$. The
residual-weighted analysis below follows the same template with
probabilities that depend on the current residual. The worst-case question
is whether this bias can ever weaken the guarantee relative to uniform
sampling, and the sharpened question is how much it helps when the
residual is non-uniform.

\subsection{Component-Selection Adaptivity Spectrum}
\label{sect:sampling-adaptivity-spectrum}

Single-component relaxation methods differ not only in the update formula, but
also in the rule used to choose the next component, row, block, or sketch.  This
choice ranges from completely oblivious rules, whose probabilities are fixed
before the iteration begins, to fully greedy rules, which compute the current
residual and select the component that promises the largest immediate progress.
The methods considered in this paper sit between these extremes: they use the
current residual to bias a randomized selection rule, but avoid the explicit
global maximization or sorting step associated with classical greedy relaxation.

At the non-adaptive end of the spectrum, the sampling distribution is fixed
before the iteration begins and never reacts to the current residual.
Strohmer and Vershynin's randomized Kaczmarz method samples rows with
probabilities proportional to squared row norms~\cite{strohmer2009randomized},
while Leventhal and Lewis analyze randomized coordinate descent and
Gauss--Seidel-type methods with probabilities chosen from static problem
data~\cite{leventhal2010randomized}. Related importance-sampling ideas appear
throughout randomized coordinate descent~\cite{nesterov2012efficiency} and
the sketch-and-project framework~\cite{gower2015randomized}. These methods
are adaptive only in the weak sense that their distributions may reflect
matrix geometry. The randomized asynchronous Jacobi analysis of
Avron, Druinsky, and Gupta~\cite{avron2015revisiting} also belongs to this
class, using uniform random selection to obtain convergence bounds under
bounded-delay and independence assumptions.

At the opposite end of the spectrum are greedy relaxation methods, which
inspect the current residual and select the component of largest residual
magnitude, which for constant diagonal is also the component of
largest immediate progress. The classical Southwell or Gauss--Southwell
rule~\cite{southwell1946relaxation} and Motzkin's row-action
relaxation~\cite{agmon1954relaxation,motzkin1954relaxation} are the
historical templates; Nutini et al.~\cite{nutini2015coordinate} give a
modern coordinate-descent analysis showing that Gauss--Southwell-type
rules can converge faster than random selection except in limiting cases,
and related analyses appear in~\cite{griebel2012greedy,frommer2023convergence}.
Greedy rules are attractive because they directly target the most productive
component, but they require global residual comparisons or rank maintenance,
which can be expensive and especially awkward in parallel or asynchronous
settings.

Between these extremes lie hybrid methods that use partial residual
information. Sample-then-greedy methods like the Sampling Kaczmarz--Motzkin
algorithm of De Loera et al.~\cite{deloera2017sampling} (sharpened
by Haddock and Ma~\cite{haddock2021greed}) sample a random subset and apply
greedy selection within it. Threshold-based methods such as Bai and Wu's
greedy randomized Kaczmarz~\cite{bai2018greedyk,bai2018relaxed} and the
analogous greedy randomized coordinate descent for least squares
problems~\cite{bai2019greedy} use the residual to form an eligible set
and then sample within that set using residual-dependent weights. The
adaptive sketch-and-project framework of Gower et
al.~\cite{gower2021adaptive} provides a unifying perspective on adaptive
sampling, including max-distance, proportional, and capped sampling
variants; these correspond roughly to greedy, residual-weighted, and
thresholded selection in our taxonomy, and the proportional rule (sampling
in proportion to the sketched loss, which for coordinate descent on an SPD
system is $r_j^2/A_{jj}$) is, for constant-diagonal $A$, the $\ell = 2$ member
of the family studied here. The continuous family itself has been analyzed 
for the Kaczmarz method
by Steinerberger~\cite{steinerberger2021weighted}, whose weighted
randomized Kaczmarz method selects row $i$ with probability proportional
to $|r_i|^p$ for any $0 < p < \infty$ and de-randomizes to
maximal-residual selection as $p \to \infty$; a closely related weighting
appears as Algorithm~3 of Jiang, Wu, and Jiang~\cite{jiang2020kaczmarz}.
The synchronous analysis of Section~\ref{sect:baseline-residual-weighted}
is the coordinate-descent (SPD Jacobi) counterpart of these results: for 
constant-diagonal $A$, the bound of Theorem~\ref{thm:sharp-power} is the 
proportional-sampling bound of~\cite{gower2021adaptive} with its 
variance factor written as the IPR, and Theorem~\ref{thm:baseline-power} 
parallels the worst-case comparison of~\cite{steinerberger2021weighted}.
Finally, Parallel and Distributed
Southwell~\cite{wolfsonpou2016reducing,wolfsonpou2017distributed} adapt
greedy selection for communication reduction by localizing the maximality
test. Apart from the proportional sketch-and-project rule~\cite{gower2021adaptive}
and weighted Kaczmarz~\cite{steinerberger2021weighted,jiang2020kaczmarz},
none of these methods uses a continuous residual-biased probability
distribution over all components, and neither of those two has been
analyzed in the asynchronous stationary setting.

The present work occupies a different point on the spectrum: the
continuous family $P(k=j) \propto |r_j|^\ell$ of Section~\ref{sec:intro},
which is residual-adaptive but needs no sorting, thresholding, or eligible
set, and interpolates explicitly between uniform selection and
Gauss--Southwell. Our earlier work on residual-biased randomized
asynchronous solvers~\cite{coleman2019enhancing,jensen2020implementing} is the closest
predecessor: it showed empirically that residual-informed selection can
improve asynchronous solver performance, but used rank-based distributions
that require periodic sorting or rank maintenance. The present work
replaces that discrete mechanism with a continuous one and supplies the
convergence analysis. Table~\ref{tab:spectrum} summarizes the spectrum.

\begin{table}[ht]
\centering
\footnotesize
\renewcommand{\arraystretch}{1.50}
\begin{tabular}{@{}p{0.20\linewidth}p{0.3\linewidth}p{0.40\linewidth}@{}}
\toprule
\textbf{Selection rule} & \textbf{Representative methods} & \textbf{Residual adaptivity} \\
\midrule
Fixed sampling
&
Randomized Kaczmarz \cite{strohmer2009randomized};
randomized CD
\cite{leventhal2010randomized,nesterov2012efficiency};
uniform asynchronous Jacobi \cite{avron2015revisiting}
&
Probabilities are fixed before the iteration begins.
The current residual does not affect the sampling rule.
\\

Sample-then-greedy
&
Sampling Kaczmarz--Motzkin (SKM)
\cite{deloera2017sampling};
improved SKM analysis \cite{haddock2021greed}
&
A random subset is sampled first; the final choice is greedy within that subset.
\\

Thresholded adaptive sampling
&
Greedy randomized Kaczmarz and CD \cite{bai2018greedyk,bai2018relaxed,bai2019greedy};
capped sketch-and-project sampling \cite{gower2021adaptive}
&
The residual determines an eligible set or capped set, and sampling is then
performed within that set using residual-dependent weights.
\\

Greedy selection
&
Southwell relaxation \cite{southwell1946relaxation};
Motzkin relaxation \cite{agmon1954relaxation,motzkin1954relaxation};
Gauss--Southwell CD \cite{nutini2015coordinate};
rule-based greedy approaches \cite{griebel2012greedy,frommer2023convergence}
&
The selected component maximizes a residual-based progress measure.
\\

Parallel greedy selection
&
Parallel Southwell and Distributed Southwell
\cite{wolfsonpou2016reducing,wolfsonpou2017distributed}
&
Greedy selection is localized to reduce communication.
\\

Rank-based residual-biased sampling
&
Residual-biased randomized asynchronous solvers
\cite{coleman2019enhancing}
&
Components are ranked by current residual magnitude and a non-uniform distribution is used to sample.
\\

Continuous residual-weighted sampling
&
This work; adaptive sketch-and-project proportional sampling
\cite{gower2021adaptive}; weighted Kaczmarz
\cite{steinerberger2021weighted,jiang2020kaczmarz}
&
All components eligible, probabilities vary directly and smoothly with current residual.  Parameter(s) give explicit limits from uniform
sampling to Gauss--Southwell selection.
\\
\bottomrule
\end{tabular}
\caption{Component-selection adaptivity spectrum: selection rules,
representative methods, and the role of the current residual in the
selection.}
\label{tab:spectrum}
\end{table}

\section{Residual-Weighted Sampling}
\label{sect:baseline-residual-weighted}

We establish that residual-weighted sampling matches the uniform bound in the worst case. Theorem~\ref{thm:baseline-power} has a Kaczmarz counterpart in~\cite{steinerberger2021weighted}; it is recorded here with a short proof because it serves as the starting point for the sharpened bounds in Section~\ref{sec:sharp}. Uniform sampling (on the nonzero residual components) emerges as the limit $\ell \to 0^+$.

\begin{theorem}[Baseline bound, power-weighted sampling]
\label{thm:baseline-power}
Under power-weighted sampling $P(k=j) = |r_j^{(i)}|^\ell / \sum_m |r_m^{(i)}|^\ell$
with any $\ell > 0$:
$$
\E[\Anorm{e^{(i+1)}}^2 \mid x^{(i)}] \leq \left(1 - \frac{\lmin(A)}{n \cdot A_{\max}}\right) \Anorm{e^{(i)}}^2.
$$
\end{theorem}

\begin{proof}
From~\eqref{eq:expected-progress-general}, the expected progress is
$\E[\mathrm{Progress}] = \sum_j P(j)\,(r_j^{(i)})^2/A_{jj}$.
Bounding $A_{jj} \leq A_{\max}$:
$$
\E[\mathrm{Progress}] \;\geq\;
\frac{1}{A_{\max}}\cdot\frac{\sum_j |r_j^{(i)}|^{\ell+2}}{\sum_m |r_m^{(i)}|^\ell}
\;=\;
\frac{1}{A_{\max}}\cdot\frac{\|r^{(i)}\|_{\ell+2}^{\ell+2}}{\|r^{(i)}\|_\ell^\ell}.
$$
It remains to show that
\begin{equation}
\label{eq:norm-ratio-bound}
\frac{\|r\|_{\ell+2}^{\ell+2}}{\|r\|_\ell^\ell} \;\geq\; \frac{\|r\|_2^2}{n}
\qquad \text{for every nonzero } r \in \R^n \text{ and every } \ell > 0.
\end{equation}
Define $p_j = r_j^2/\|r\|_2^2$, which is a probability distribution on $\{1,\ldots,n\}$.
Then $|r_j|^{2k} = p_j^k\,\|r\|_2^{2k}$ for every $k \geq 0$, so
\[
\|r\|_\ell^\ell = \|r\|_2^\ell \sum_j p_j^{\ell/2},
\qquad
\|r\|_{\ell+2}^{\ell+2} = \|r\|_2^{\ell+2} \sum_j p_j^{(\ell+2)/2},
\]
and~\eqref{eq:norm-ratio-bound} reduces to
\begin{equation}
\label{eq:cheb-target}
n \sum_j p_j^{(\ell+2)/2} \;\geq\; \sum_j p_j^{\ell/2},
\end{equation}
since substituting the norm identities into~\eqref{eq:norm-ratio-bound}
and simplifying by $\|r\|_2^2$ gives
$\sum_j p_j^{(\ell+2)/2} / \sum_j p_j^{\ell/2} \geq 1/n$.
Setting $\alpha = \ell/2 \geq 0$, both sequences $(p_j)_{j=1}^n$ and $(p_j^\alpha)_{j=1}^n$ are
similarly ordered (the function $t \mapsto t^\alpha$ is non-decreasing on $[0,1]$ for $\alpha \geq 0$).
Chebyshev's sum inequality\footnote{For real sequences
$a_1 \leq \cdots \leq a_n$ and $b_1 \leq \cdots \leq b_n$,
$n\sum_i a_i b_i \geq (\sum_i a_i)(\sum_i b_i)$; see e.g.
\cite[Theorem 43]{hardy1952inequalities}.} therefore gives
\[
n \sum_j p_j \cdot p_j^\alpha
\;\geq\;
\Bigl(\sum_j p_j\Bigr)\Bigl(\sum_j p_j^\alpha\Bigr)
\;=\; \sum_j p_j^\alpha,
\]
using $\sum_j p_j = 1$. This is~\eqref{eq:cheb-target}, hence~\eqref{eq:norm-ratio-bound}.
Combining with $\|r^{(i)}\|_2^2 \geq \lmin(A)\,\Anorm{e^{(i)}}^2$ yields the claim.
\end{proof}

Theorem~\ref{thm:baseline-power} gives the \emph{same} worst-case bound as uniform sampling (Proposition~\ref{thm:uniform}), and with it convergence in expectation for every $\ell > 0$. Residual-weighting therefore retains the uniform worst-case guarantee, but these baseline bounds provide no incentive to use it. Section~\ref{sec:sharp} shows that part of the gap between these bounds and the empirical convergence rate is captured by the inverse participation ratio.

\subsection{Sharpened Bounds via the Inverse Participation Ratio}
\label{sec:sharp}

\begin{definition}[Inverse Participation Ratio]
\label{def:ipr}
For a nonzero vector $r \in \R^n$, the \emph{inverse participation ratio} (IPR) is
\begin{equation}
\label{eq:ipr}
\nu^2(r) = \frac{n \cdot \norm{r}{4}^4}{\norm{r}{2}^4} = \frac{n \sum_j r_j^4}{\left(\sum_j r_j^2\right)^2}.
\end{equation}
\end{definition}

The squared form reflects the fact that the quantity naturally enters bounds
as $\nu^2 = (\sqrt{n}\,\|r\|_4^2/\|r\|_2^2)^2$; writing it unsquared as
$\nu = \sqrt{n}\,\|r\|_4^2/\|r\|_2^2$ is occasionally useful (e.g., in
Cauchy--Schwarz bounds) but we use $\nu^2$ by default. The IPR has the
following properties.

\begin{lemma}[Properties of the IPR]
\label{lem:ipr-properties}
For any nonzero $r \in \R^n$:
\begin{enumerate}[label=(\roman*)]
    \item $\nu^2(r) \geq 1$, with equality if and only if all $|r_j|$ are equal.
    \item $\nu^2(r) \leq n$, with equality if and only if $r$ has exactly one nonzero component.
    \item $\nu^2(r)$ is invariant under uniform scaling: $\nu^2(cr) = \nu^2(r)$ for $c \neq 0$.
    \item $\nu^2(r)$ depends only on the \emph{shape} of the distribution of $|r_j|^2$, not its magnitude.
\end{enumerate}
\end{lemma}

\begin{proof}
\textit{(i)} By Cauchy--Schwarz applied to the vectors
$(1, \ldots, 1)$ and $(r_1^2, \ldots, r_n^2)$,
\[
\Bigl(\sum_j r_j^2\Bigr)^2 = \Bigl(\sum_j 1 \cdot r_j^2\Bigr)^2
\;\leq\; n \sum_j r_j^4,
\]
giving $\nu^2(r) \geq 1$. Equality in Cauchy--Schwarz holds iff
$(r_1^2, \ldots, r_n^2)$ is a scalar multiple of $(1, \ldots, 1)$,
i.e., iff all $|r_j|$ are equal.

\textit{(ii)} For any $r \in \R^n$,
\[
\sum_j r_j^4
\;\leq\; (\max_j r_j^2) \sum_j r_j^2
\;\leq\; \Bigl(\sum_j r_j^2\Bigr)^2,
\]
where the first inequality bounds each $r_j^2$ in $r_j^4 = r_j^2 \cdot r_j^2$
by the maximum, and the second uses $\max_j r_j^2 \leq \sum_j r_j^2$ for
non-negative terms. Hence $\nu^2(r) \leq n$. Equality throughout requires
$\max_j r_j^2 = \sum_j r_j^2$, which (for non-negative summands) holds iff
exactly one $r_j$ is nonzero.

\textit{(iii)} For $c \neq 0$,
$\nu^2(cr) = n\,c^4 \|r\|_4^4 / (c^2 \|r\|_2^2)^2 = \nu^2(r)$.

\textit{(iv)} Follows from (iii): the normalized squared residual
$p_j = r_j^2/\|r\|_2^2$ is invariant under uniform scaling, and a direct
computation gives $\nu^2(r) = n\sum_j p_j^2$, which depends only on $(p_j)$.
\end{proof}

The particular form of $\nu^2$ is not a design choice but a consequence of
the proof. The one-step progress identity~\eqref{eq:progress-identity}
combined with power-weighted sampling $P(k=j) = r_j^2/\norm{r}{2}^2$
produces the ratio $\norm{r}{4}^4/\norm{r}{2}^2$ as a natural progress
lower bound (derivation in the proof of Theorem~\ref{thm:sharp-power}).
The question of what this ratio equals is forced: to compare with uniform
sampling we need to express it as a multiple of $\norm{r}{2}^2/n$, and
the multiple that appears is exactly $\nu^2$.
Other scalar measures of non-uniformity are possible, but this one is forced by
the expected-progress identity; attempts to use
$\|r\|_\infty / (\|r\|_2/\sqrt n)$ or $\|r\|_1 / (\sqrt n \|r\|_2)$
produce measures of non-uniformity but do not arise from the proof
and do not give the sharpened bound.

The IPR admits three complementary interpretations.
As a \emph{localization or participation} measure, it is a standard tool
in physics for quantifying how many sites or degrees of freedom meaningfully
participate in a state~\cite{evers2008anderson}: $\nu^2 \approx 1$ means
delocalized, $\nu^2 \gg 1$ means localized on a few components.
As an \emph{effective support}, treating $p_j = r_j^2/\|r\|_2^2$ as a
probability distribution gives $\nu^2(r) = n\sum_j p_j^2$, so
$n/\nu^2(r) = 1/\sum_j p_j^2 = e^{H_2(p)}$ is the order-2 effective
support of $p$, where $H_2(p) = -\log\sum_j p_j^2$ is the R\'enyi-2
entropy~\cite{renyi1961measures}; a residual spread uniformly over $5$
components has $\nu^2 = n/5$.
Finally, a direct computation gives $\nu^2(r) = 1 + \mathrm{Var}(r_j^2)/\bar{r^2}^{\,2}$
with mean and variance taken over a uniform index $j$, so $\nu^2 - 1$ is the
\emph{squared coefficient of variation} of $(r_j^2)$ and $\nu^2 > 1$ iff
the squared residuals are not all equal. For our purposes:
$\nu^2 \approx 1$ means biased sampling has little advantage, while
$\nu^2 \gg 1$ means a few components dominate and biased sampling can
concentrate effort where it matters. 

\subsection{Sharpened Convergence for Power-Weighted Jacobi ($\ell = 2$)}

While the baseline bounds in the previous section apply to multiple weighting parameters, the choice $\ell=2$ is the natural one for analysis. When sampling probabilities are proportional to the squared residual components, $P(k=j) = r_j^2/\|r\|_2^2$, they form a distribution whose expected progress naturally scales with the variance of those squared residuals. This choice causes the non-uniformity of the residual to drop directly out of the one-step expectation as a clean scalar multiplier, allowing the inverse participation ratio to act as the exact bridge between the empirical spread of the residual and the theoretical convergence rate.

\begin{theorem}[Sharpened Power-Weighted Convergence]
\label{thm:sharp-power}
Let $A \in \R^{n \times n}$ be SPD. Under power-weighted randomized Jacobi with $\ell = 2$ (probability $\propto r_j^2$):
\begin{equation}
\label{eq:sharp-power}
\E[\Anorm{e^{(i+1)}}^2 \mid x^{(i)}] \leq \left(1 - \frac{\nu_i^2 \cdot \lmin(A)}{n \cdot A_{\max}}\right) \Anorm{e^{(i)}}^2
\end{equation}
where $\nu_i^2 = \nu^2(r^{(i)}) = n\norm{r^{(i)}}{4}^4 / \norm{r^{(i)}}{2}^4$ is the IPR of the current residual (if $r^{(i)} = 0$ the iteration has converged, and we set $\nu_i^2 = 1$).

In particular, the guaranteed expected per-step progress is amplified by $\nu_i^2$ relative to uniform sampling, replacing the uniform contraction factor of $(1-\alpha)$ by $(1 - \nu_i^2\alpha)$ where $\alpha = \lmin(A)/(nA_{\max})$. This factor satisfies $0 \leq 1 - \nu_i^2\alpha \leq 1 - \alpha < 1$, so the iteration converges in expectation: $\E[\Anorm{e^{(m)}}^2] \leq (1-\alpha)^m\,\Anorm{e^{(0)}}^2 \to 0$ as $m \to \infty$.
\end{theorem}

\begin{proof}
Write $\nu^2 = \nu^2(r^{(i)})$ throughout. From~\eqref{eq:expected-progress-general},
the expected progress under $P(k=j) = r_j^2/\|r\|_2^2$ is
\begin{align}
\E[\mathrm{Progress} \mid x^{(i)}]
&= \sum_{j=1}^n \frac{r_j^2}{\|r\|_2^2}\cdot \frac{r_j^2}{A_{jj}}
 = \frac{1}{\|r\|_2^2}\sum_{j=1}^n \frac{r_j^4}{A_{jj}}.
\end{align}
Using $A_{jj} \leq A_{\max}$:
\begin{equation}
\E[\mathrm{Progress} \mid x^{(i)}]
\;\geq\; \frac{1}{A_{\max}}\cdot \frac{\|r\|_4^4}{\|r\|_2^2}.
\end{equation}
By the IPR identity $\|r\|_4^4 = (\nu^2/n)\,\|r\|_2^4$
(Definition~\ref{def:ipr}):
\begin{equation}
\E[\mathrm{Progress} \mid x^{(i)}]
\;\geq\; \frac{\nu^2}{n A_{\max}}\,\|r\|_2^2.
\end{equation}
Applying $\|r\|_2^2 \geq \lmin(A)\Anorm{e^{(i)}}^2$ from~\eqref{eq:residual-energy-bounds}:
\[
\E[\mathrm{Progress} \mid x^{(i)}]
\;\geq\; \frac{\nu^2\,\lmin(A)}{n\,A_{\max}}\,\Anorm{e^{(i)}}^2,
\]
and~\eqref{eq:sharp-power} follows from $\E[\Anorm{e^{(i+1)}}^2 \mid x^{(i)}] =
\Anorm{e^{(i)}}^2 - \E[\mathrm{Progress} \mid x^{(i)}]$.
For the range of the factor, $1 \leq \nu_i^2 \leq n$ by
Lemma~\ref{lem:ipr-properties} and $\lmin(A) \leq A_{jj} \leq A_{\max}$ for
every $j$, so $\alpha \leq \nu_i^2\alpha \leq \lmin(A)/A_{\max} \leq 1$.
Taking expectations in~\eqref{eq:sharp-power} and iterating gives the bound
on $\E[\Anorm{e^{(m)}}^2]$.
\end{proof}

\begin{remark}[Exactness for constant diagonal]
\label{rem:exact-constant-diag}
The only inequality in the proof that involves the sampling rule is
$A_{jj} \leq A_{\max}$. If $A$ has constant diagonal $d$, as for the
finite-difference operators of Section~\ref{sec:experiments}, then
$\E[\mathrm{Progress} \mid x^{(i)}] = \nu_i^2\,\norm{r}{2}^2/(nd)$, which is
exactly $\nu_i^2$ times the expected progress $\norm{r}{2}^2/(nd)$ of
uniform sampling from the same iterate. In that case the amplification by
$\nu_i^2$ holds for the expected progress itself and not only for the
guaranteed coefficient.
\end{remark}

Setting $\nu_i^2 = 1$ recovers the baseline bound of
Theorem~\ref{thm:baseline-power} exactly, so the sharpened bound is
never weaker than the baseline and is strictly stronger whenever
$\nu_i > 1$. In an implementation, $\nu_i^2$ can be accumulated in the pass
over the residual already needed for
the sampling step itself, so the IPR doubles as a per-iteration
diagnostic that reports in real time the factor by which the guaranteed
progress coefficient exceeds that of uniform sampling. For concreteness, on the
\texttt{poisson} problem of Section~\ref{sec:experiments}
($n = 16{,}384$, $\lambda_{\min}/A_{\max} \approx 3\times10^{-4}$) the
steady-state IPR $\nu^2 \approx 5.7$ amplifies the uniform-sampling
progress coefficient by roughly a factor of six, and the
early-iteration transient (where $\nu^2$ is in the tens to hundreds)
contributes a large share of the actual error reduction;
Figure~\ref{fig:combined_sync_bounds} shows that the realized
convergence on this problem is in fact faster than the steady-state
bound alone predicts.


The one-step result (Theorem~\ref{thm:sharp-power}) holds unconditionally. To obtain a sharper global (multi-step) rate, we need the IPR to remain bounded below across iterations.

\begin{corollary}[Global Rate Under Sustained Non-Uniformity]
\label{cor:global-rate}
If $\nu_i^2 \geq \bar{\nu}^2 > 1$ for all $i = 0, 1, \ldots, m-1$, regardless of which components are sampled, then
$$
\E[\Anorm{e^{(m)}}^2] \leq \left(1 - \frac{\bar{\nu}^2 \cdot \lmin(A)}{n \cdot A_{\max}}\right)^m \Anorm{e^{(0)}}^2.
$$
\end{corollary}

Whether $\nu_i^2 \geq \bar{\nu}^2$ holds in practice depends on the problem. We investigate this numerically in Section~\ref{sec:experiments}. Proving it analytically for specific matrix classes (e.g., discretized PDEs with localized forcing) is an interesting open direction.
Without the sustained non-uniformity assumption, we can still obtain a global bound by using the IPR as a per-step weight.

\begin{corollary}[Global Rate via Geometric Mean of IPR]
\label{cor:geometric-mean}
Assume that $A$ is not a multiple of the identity. After $m$ iterations,
$$
\E\!\left[\,\Anorm{e^{(m)}}^2 \prod_{i=0}^{m-1}\left(1 - \frac{\nu_i^2 \cdot \lmin(A)}{n \cdot A_{\max}}\right)^{-1}\right] \leq \Anorm{e^{(0)}}^2.
$$
\end{corollary}

\begin{proof}
The values $\nu_i^2$ depend on the components sampled so far, so the product is itself random and must stay inside the expectation. Every factor is positive, because $\nu_i^2 \leq n$ and $\lmin(A) < A_{\max}$ unless $A$ is a multiple of the identity. Once the first $m-1$ steps are given, the product is fixed, and Theorem~\ref{thm:sharp-power} bounds the expectation of $\Anorm{e^{(m)}}^2$ over the last step by the last factor times $\Anorm{e^{(m-1)}}^2$. Dividing by the product and taking expectations shows that the left-hand side does not increase with $m$; for $m = 0$ it equals $\Anorm{e^{(0)}}^2$.
\end{proof}

\noindent
In words: measured against the product of the one-step factors realized along the run, the error is bounded in expectation. Section~\ref{sec:experiments} plots this product as the IPR-based prediction for each run.
Taking logarithms, the effective convergence rate is governed by the arithmetic mean of $\log(1 - \nu_i^2 \lmin/(nA_{\max}))$, which is dominated by the average IPR.

The sharpening is not specific to $\ell = 2$. For every $\ell > 0$, the first display in the proof of Theorem~\ref{thm:baseline-power} gives the bound~\eqref{eq:sharp-power} with $\nu_i^2$ replaced by the ratio $n\,\|r^{(i)}\|_{\ell+2}^{\ell+2}/(\|r^{(i)}\|_\ell^\ell\,\|r^{(i)}\|_2^2)$, which is at least $1$ by~\eqref{eq:norm-ratio-bound} and reduces to $\nu_i^2$ at $\ell = 2$. We do not study this generalized ratio further; the asynchronous analysis of Section~\ref{sec:async} is for $\ell = 2$.

\section{Asynchronous Power-Weighted Jacobi}
\label{sec:async}

The synchronous result in Theorem~\ref{thm:sharp-power} sharpens the
uniform-sampling bound by exactly the IPR, giving a clean theoretical
separation between the guarantees for residual-weighted and uniform
sampling at every step. A natural next question is whether the same analytical leverage
carries over to asynchronous execution, where the residual driving the
sampling distribution may be stale relative to the iterate being
updated. The question is not merely academic: residual-weighted
sampling is most appealing precisely in the high-concurrency regime
where asynchrony is unavoidable, and the synchronous bound is silent
about what happens there.

This section extends the sharpened synchronous analysis to the
asynchronous setting. We first review the Avron--Druinsky--Gupta (ADG)
framework~\cite{avron2015revisiting} for uniform asynchronous Jacobi,
then identify three specific obstacles that arise under
residual-weighted sampling and show how each can be handled under a
global bound on the IPR trajectory, for a delay sequence that is fixed in
advance. The result is a convergence
theorem structurally analogous to ADG's Theorem~2a, with the
uniform-sampling cross-term parameter $\rho$ replaced by a dynamic
parameter $\rho_{\text{dyn}}$ and with the IPR entering both the
progress coefficient and the allowed-delay window. The theorem should
be read as a first asynchronous convergence result for power-weighted
sampling rather than as a strict generalization of the uniform theory;
several aspects of the bound are provably loose and are flagged as
sharpening targets at the end of the section.


To match the notation of Section~\ref{sec:background}, we write the asynchronous iteration in component form. At step $j$ the worker reads a stale iterate $x^{(s(j))}$ (where $s(j) \leq j$ is the step at which the read completed), samples a component $k_j$ according to a probability distribution that may depend on the stale residual $\tilde r^{(j)} = b - A x^{(s(j))}$, computes the Jacobi update using the stale residual's $k_j$-th component, and commits the result to the current iterate $x^{(j)}$:
\begin{equation}
\label{eq:async-update}
x^{(j+1)} = x^{(j)} + \beta\cdot\frac{\tilde r_{k_j}^{(j)}}{A_{k_j, k_j}}\, e_{k_j}.
\end{equation}
The bounded-delay condition $j - s(j) \leq \tau$ limits how stale any read can be; $\beta \in (0, 1]$ is an optional step-size parameter, with $\beta = 1$ corresponding to a full Jacobi step. Setting $s(j) = j$ for all $j$ (no staleness) and $\beta = 1$ recovers the synchronous iteration~\eqref{eq:update}. The update is applied to $x^{(j)}$, not $x^{(s(j))}$; the read is stale but the write is current. Throughout, the read times $s(j)$ are assumed to be fixed in advance, independently of the components that are sampled.

\begin{algorithm}[ht]
\small
\caption{Asynchronous power-weighted randomized Jacobi (per-worker loop)}
\label{alg:async-power-jacobi}
\begin{algorithmic}[1]
\Require SPD matrix $A$, RHS $b$, shared iterate $x$, shared
         residual $r$, exponent $\ell > 0$, step size
         $\beta \in (0, 1]$, read mode $\in$ \{buffered (``consistent''),
         refreshed (``inconsistent'')\}
\While{not stopped (tolerance or work cap)}
  \If{read mode is consistent (buffered)}
    \State $r_{\text{local}} \gets$ bulk copy of $r$ (\texttt{memcpy}; not an atomic snapshot, components may differ in age by the copy duration)
    \Comment{Used for both sampling and the update value}
  \Else
    \State $r_{\text{local}} \gets$ componentwise copy of $r$ via \texttt{\#pragma omp atomic read}
    \Comment{Used for sampling only}
  \EndIf
  \State Sample $k \in \{1, \ldots, n\}$ with
         $P(k = j) = |r_{\text{local},j}|^\ell / \sum_m |r_{\text{local},m}|^\ell$
         (if the denominator is zero, sample $k$ uniformly; every increment is then zero)
  \If{read mode is inconsistent (refreshed)}
    \State $r_k \gets$ atomic read of $r_k$
    \Comment{Re-read immediately before forming the increment}
  \Else
    \State $r_k \gets r_{\text{local},k}$
    \Comment{Buffered value, possibly superseded by in-flight commits}
  \EndIf
  \State $\delta \gets \beta \cdot r_k / A_{kk}$
  \State Atomic update: $x_k \gets x_k + \delta$
  \State Apply the stencil update to $r$ at $k$ and its neighbors in the
         matrix graph, each entry atomically
\EndWhile
\end{algorithmic}
\end{algorithm}

For comparison with the ADG literature, our $x^{(s(j))}$ corresponds to their $x_{k(j)}$ and our $k_j$ to their direction $d_j = e_{k_j}$; our $\tilde r_{k_j}^{(j)}/A_{k_j,k_j}$ corresponds to their $\gamma_j = (x^* - x_{k(j)}, d_j)_A$ (for unit-diagonal $A$ and $d_j = e_{k_j}$, $(x^* - x_{k(j)}, e_{k_j})_A = [A(x^* - x_{k(j)})]_{k_j} = \tilde r_{k_j}^{(j)}$). Translation between the two conventions is purely notational.
Under uniform direction sampling ($P(k_j = r) = 1/n$), consistent reads, and an independent-delays assumption (ADG Assumption A-4), Avron et al.\ establish the following result.

\begin{proposition}[ADG Theorem~2a; \cite{avron2015revisiting}]
\label{thm:avron-async}
Consider iteration~\eqref{eq:async-update} with uniform sampling, $\beta = 1$, unit diagonal $A$, consistent reads, and bounded delay $\tau$. Let $\rho = \frac{1}{n}\|A\|_\infty$. Provided $2\rho\tau < 1$, for every epoch of length $m \geq 0.693\,n/\lmax$,
$$
E_m \leq \left(1 - \frac{1 - 2\rho\tau}{2\kappa}\right) E_0,
$$
where $E_j = \E[\Anorm{e^{(j)}}^2]$.
\end{proposition}

The proof architecture has three stages: a one-step progress lemma with
both lower and upper bounds on the expected squared inner product at
the stale state; an unrolled recursion that separates per-step
progress from cross-coupling damage between stale components; and an
epoch assembly that sums the recursion using the upper bound to ensure
the error does not collapse faster than the damage can accumulate.


Porting Theorem~\ref{thm:sharp-power} ($\ell = 2$) into this architecture, three specific mechanisms break.

\paragraph{Obstacle 1: The independent-delays assumption fails.} ADG Assumption A-4 requires that $s(j)$ is independent of the sequence of component choices $k_0, k_1, \ldots$. Under uniform sampling this is plausible; the scheduler's choice of when to read cannot influence which component gets sampled. Under power-weighted sampling the distribution over components at step $j$ is $P(k_j = r) = (\tilde r_r^{(j)})^2 / \|\tilde r^{(j)}\|_2^2$, which depends on the stale residual and hence on $s(j)$. The component choice and the delay are coupled through the stale iterate.

\paragraph{Obstacle 2: The cross-term bound loses its $1/n$ scaling.} In the uniform proof, for any fixed past component $k_t = l$,
$$
\E[\,|A_{l, k_j}|\,] = \frac{1}{n}\sum_r |A_{lr}| \leq \frac{1}{n}\|A\|_\infty = \rho.
$$
The $1/n$ factor is what makes $\rho = O(1/n)$ for sparse matrices. Under power-weighted sampling the analogous expectation is $\sum_r ((\tilde r_r^{(j)})^2/\|\tilde r^{(j)}\|_2^2)\, |A_{lr}|$ and the favorable $1/n$ prefactor is lost. A residual concentrated on a single component would put all its probability mass on a single matrix row, giving no $n$-dependence at all.

\paragraph{Obstacle 3: The progress lemma needs an upper bound.} Theorem~\ref{thm:sharp-power} provides a \emph{lower} bound on expected progress. The async epoch assembly also requires an \emph{upper} bound to guarantee that delayed error terms are still large enough fractions of the initial error for the sum over the epoch to accumulate.

\vspace{1em}

Our main tool for handling these obstacles is to assume a uniform bound on the IPR throughout the epoch. Together with read times that are fixed in advance, this simultaneously (a) sidesteps Obstacle 1 by conditioning on the stale state and imposing worst-case dynamics rather than tracking actual correlations, (b) partially repairs Obstacle 2 by allowing Cauchy--Schwarz to decouple residual from matrix, and (c) gives us a two-sided lemma that resolves Obstacle 3 directly.

\begin{assumption}[Uniform IPR trajectory]
\label{assump:global-ipr}
There exist constants $1 \leq \nu_{\min}^2 \leq \nu_{\max}^2 \leq n$ such that for every step $j$ in the epoch with $\tilde r^{(j)} \neq 0$, regardless of which components are sampled,
$$
\nu_{\min}^2 \leq \nu^2(\tilde r^{(j)}) \leq \nu_{\max}^2.
$$
\end{assumption}

Assumption~\ref{assump:global-ipr} requires a priori bounds on the full
trajectory of IPRs, including during early iterations when the
residual non-uniformity is changing most rapidly (see
Section~\ref{sec:experiments}); we refer to the interval
$[\nu_{\min}^2, \nu_{\max}^2]$ as the IPR \emph{envelope}. In an
implementation one could compute
$\nu^2(\tilde r^{(j)})$ online and check the assumption post hoc, but
it is an input to the theorem rather than an output. The trivial envelope
$\nu_{\min}^2 = 1$, $\nu_{\max}^2 = n$ always satisfies it
(Lemma~\ref{lem:ipr-properties}), so the results below are never vacuous,
but they are informative only when $\nu_{\max}^2$ is much smaller than $n$,
and we do not know a class of problems for which such an envelope can be
established in advance.
Section~\ref{sec:experiments} examines the assumption empirically: on all
three test problems the sustained $\nu^2$ lies between about $3.5$ and $6$,
and on \texttt{poisson} it varies by less than $2\%$ across two decades of
thread count (Table~\ref{tab:hw-scaling}). These observations indicate the
envelope that the sustained phase would support; they do not verify the
assumption, which must hold on every run and must also cover the start-up
transient (for a point-source right-hand side $\nu^2(r^{(0)}) = n$, so an
epoch that begins at iteration $0$ has $\nu_{\max}^2 = n$). What we get in
exchange is a mechanical
path to a convergence theorem with the right structural form; closing
the gap between the assumption and the dynamics analytically is the
main open problem.

With $\ell = 2$ power-weighting at the stale state, letting $\tilde e^{(j)} = x^{(s(j))} - x^*$ and taking the expectation over the choice of $k_j$ for a given stale residual,
$$
\E[(\tilde e^{(j)}, e_{k_j})_A^2] = \sum_{r=1}^n \frac{(\tilde r_r^{(j)})^2}{\|\tilde r^{(j)}\|_2^2}\, (\tilde r_r^{(j)})^2 = \frac{\|\tilde r^{(j)}\|_4^4}{\|\tilde r^{(j)}\|_2^2} = \frac{\nu^2(\tilde r^{(j)})}{n}\|\tilde r^{(j)}\|_2^2.
$$
Applying $\lmin \Anorm{\tilde e^{(j)}}^2 \leq \|\tilde r^{(j)}\|_2^2 \leq \lmax \Anorm{\tilde e^{(j)}}^2$ gives a two-sided control on the conditional expected progress.

\begin{lemma}[Dynamic one-step lemma]
\label{lem:dynamic-async-1}
Under power-weighted sampling with $\ell = 2$, for a given stale residual $\tilde r^{(j)} \neq 0$,
$$
\frac{\nu^2(\tilde r^{(j)})\,\lmin}{n}\Anorm{\tilde e^{(j)}}^2 \;\leq\; \E[(\tilde e^{(j)}, e_{k_j})_A^2] \;\leq\; \frac{\nu^2(\tilde r^{(j)})\,\lmax}{n}\Anorm{\tilde e^{(j)}}^2.
$$
Under Assumption~\ref{assump:global-ipr}, the same bounds hold with $\nu^2(\tilde r^{(j)})$ replaced by $\nu_{\min}^2$ (lower) and $\nu_{\max}^2$ (upper).
\end{lemma}

For the cross-term, under power-weighted sampling, for past component $k_t = l$ fixed,
$$
\E[\,|A_{l, k_j}|\,] = \sum_r \frac{(\tilde r_r^{(j)})^2}{\|\tilde r^{(j)}\|_2^2}\,|A_{lr}|
\leq \sqrt{\sum_r \left(\frac{(\tilde r_r^{(j)})^2}{\|\tilde r^{(j)}\|_2^2}\right)^{\!2}}\cdot\|A_{l,:}\|_2
= \frac{\nu(\tilde r^{(j)})}{\sqrt n}\,\|A_{l,:}\|_2
$$
by Cauchy--Schwarz. Defining $\rho_2 := \max_l \|A_{l,:}\|_2$ and using Assumption~\ref{assump:global-ipr},
\begin{equation}
\label{eq:async-row-bound}
\E[\,|A_{l, k_j}|\,] \leq \frac{\nu_{\max}\,\rho_2}{\sqrt n} =: \rho_{\text{dyn}}.
\end{equation}
For a sparse matrix with $c$ nonzeros per row and bounded entries, $\rho \approx cM/n = O(1/n)$ in the uniform case, while $\rho_{\text{dyn}} = \nu_{\max} M\sqrt c/\sqrt n = O(\nu_{\max}/\sqrt n)$ here. Even setting $\nu_{\max} = 1$ (flat residual) gives $\rho_{\text{dyn}} = O(1/\sqrt n)$, which is strictly worse than the ADG bound of $O(1/n)$. The dynamic analysis does not recover the uniform result as a special case; the Cauchy--Schwarz step throws away $\ell_1$-style summability that the uniform analysis uses directly. This is the most important sharpening target (see Section~\ref{sec:limitations}).

\subsection{Convergence Theorem}

\begin{theorem}[Asynchronous power-weighted Jacobi, $\ell = 2$]
\label{thm:async-power-weighted}
Consider the asynchronous iteration~\eqref{eq:async-update} with step-size $\beta = 1$, unit diagonal $A$, power-weighted sampling ($\ell = 2$, using the stale residual), consistent reads, and bounded delay $\tau$ with read times fixed in advance. Under Assumption~\ref{assump:global-ipr}, let
$$
\rho_{\text{dyn}} = \frac{\nu_{\max}\,\rho_2}{\sqrt n}, \qquad \tilde\nu_\tau = 1 - 2\rho_{\text{dyn}}\tau.
$$
Provided $\tau < \sqrt n / (2\nu_{\max}\rho_2)$, for every epoch of length $m \geq (\ln 2)\,n/(\nu_{\max}^2\lmax)$ starting from the initial iterate,
$$
E_m \leq \left(1 - \frac{\nu_{\min}^2}{\nu_{\max}^2} \cdot \frac{\tilde\nu_\tau}{2\kappa\,(1 + 2\rho_{\text{dyn}}\tau)}\right) E_0.
$$
\end{theorem}

\begin{proof}
The proof adapts the three-stage architecture of ADG~\cite{avron2015revisiting} (Theorem~2) to the power-weighted setting: one-step expansion, telescoped recursion with cross-term decoupling, and epoch assembly via the ``error cannot collapse'' argument. The substitutions are: Lemma~\ref{lem:dynamic-async-1} replaces the uniform one-step lemma, and the cross-term parameter $\rho_{\text{dyn}}$ from Section~\ref{sec:async} replaces ADG's $\rho$. We carry out all of the steps, because the epoch assembly must also account for the interference terms.

\smallskip\noindent\textit{Step 1: One-step error identity.} With unit-diagonal $A$ and $\beta = 1$, the update~\eqref{eq:async-update} gives $e^{(j+1)} = e^{(j)} + \tilde r_{k_j}^{(j)} e_{k_j}$. The accumulated in-flight updates between read time $s(j)$ and write time $j$ give
\[
r^{(j)} = \tilde r^{(j)} - \sum_{l \in J_j} \tilde r_{k_l}^{(l)} A e_{k_l}, \qquad J_j = \{l : s(j) \leq l < j\}, \qquad |J_j| \leq \tau,
\]
where $J_j$ is the in-flight set and the bound on its size is the bounded-delay condition. Expanding $\Anorm{e^{(j+1)}}^2$ in the $A$-inner product, using $A_{k_j,k_j} = 1$ and substituting $(Ae^{(j)})_{k_j} = -r_{k_j}^{(j)}$,
\begin{equation}
\label{eq:async-onestep-proof}
\Anorm{e^{(j+1)}}^2 = \Anorm{e^{(j)}}^2 - (\tilde r_{k_j}^{(j)})^2 + 2 \sum_{l \in J_j} A_{k_j, k_l}\,\tilde r_{k_j}^{(j)}\,\tilde r_{k_l}^{(l)}.
\end{equation}
The negative term is the stale progress; the sum is the interference from in-flight updates.

\smallskip\noindent\textit{Step 2: AM--GM decoupling.} Apply $2|ab| \leq a^2 + b^2$ to each cross-coupling term:
\begin{equation}
\label{eq:async-amgm}
\left|2 A_{k_j, k_l}\,\tilde r_{k_j}^{(j)}\,\tilde r_{k_l}^{(l)}\right| \leq |A_{k_j, k_l}| \left[(\tilde r_{k_j}^{(j)})^2 + (\tilde r_{k_l}^{(l)})^2\right].
\end{equation}

\smallskip\noindent\textit{Step 3: Conditional expectations.} Let $\mathcal{F}_j$ denote the $\sigma$-algebra of all information up to (but not including) the random choice of $k_j$, so that $\tilde r^{(j)}$, $J_j$, and $\{(k_l, \tilde r^{(l)}) : l < j\}$ are $\mathcal{F}_j$-measurable. We require three conditional bounds.

\textit{(a) Progress (lower).} Lemma~\ref{lem:dynamic-async-1} gives
\[
\E\!\left[(\tilde r_{k_j}^{(j)})^2 \mid \mathcal{F}_j\right] \;\geq\; \frac{\nu_{\min}^2\,\lmin}{n}\,\Anorm{\tilde e^{(j)}}^2.
\]

\textit{(b) Stale interference factor.} Since $(\tilde r_{k_l}^{(l)})^2$ is $\mathcal{F}_j$-measurable it pulls out of the expectation, and the matrix entry is bounded by the cross-term inequality preceding Theorem~\ref{thm:async-power-weighted}:
\[
\E\!\left[|A_{k_j, k_l}|\,(\tilde r_{k_l}^{(l)})^2 \mid \mathcal{F}_j\right] = (\tilde r_{k_l}^{(l)})^2\,\E\!\left[|A_{k_j, k_l}| \mid \mathcal{F}_j\right] \;\leq\; \rho_{\text{dyn}}\,(\tilde r_{k_l}^{(l)})^2.
\]

\textit{(c) Current interference factor.} Here both the matrix entry and the residual square depend on $k_j$, so the argument of (b) does not apply; instead we also average over the in-flight choice $k_l$. Since $l \in J_j$ means $s(j) \leq l$, the stale residual $\tilde r^{(j)}$ was read before update $l$ was committed and, the read times being fixed in advance, it is $\mathcal{F}_l$-measurable, as is $\tilde r^{(l)}$. The distribution of $k_j$ is determined by $\tilde r^{(j)}$ and so does not involve $k_l$: conditional on $\mathcal{F}_l$, the choices $k_l$ and $k_j$ are independent, with probabilities $p^{(l)}_q = (\tilde r_q^{(l)})^2/\|\tilde r^{(l)}\|_2^2$ and $p^{(j)}_r = (\tilde r_r^{(j)})^2/\|\tilde r^{(j)}\|_2^2$. Hence
\begin{align*}
\E\!\left[|A_{k_j, k_l}|\,(\tilde r_{k_j}^{(j)})^2 \mid \mathcal{F}_l\right]
&= \sum_{r=1}^n p^{(j)}_r\,(\tilde r_r^{(j)})^2 \sum_{q=1}^n p^{(l)}_q\,|A_{rq}| \\
&\leq \rho_{\text{dyn}} \sum_{r=1}^n p^{(j)}_r\,(\tilde r_r^{(j)})^2
\;=\; \rho_{\text{dyn}}\,\E\!\left[(\tilde r_{k_j}^{(j)})^2 \mid \mathcal{F}_l\right],
\end{align*}
where the inner sum is bounded by the cross-term inequality~\eqref{eq:async-row-bound}, applied to the distribution of $k_l$.

Taking total expectations in~\eqref{eq:async-onestep-proof} and~\eqref{eq:async-amgm}, applying (b) and (c) to each $l \in J_j$, and using $|J_j| \leq \tau$,
\begin{equation}
\label{eq:async-condexp}
E_{j+1} \;\leq\; E_j - (1 - \rho_{\text{dyn}}\tau)\,\E\!\left[(\tilde r_{k_j}^{(j)})^2\right] + \rho_{\text{dyn}}\sum_{l \in J_j}\E\!\left[(\tilde r_{k_l}^{(l)})^2\right],
\end{equation}
where $E_j = \E[\Anorm{e^{(j)}}^2]$.

\smallskip\noindent\textit{Step 4: Telescoping and double-counting.} Sum~\eqref{eq:async-condexp} over $j = 0, \ldots, m-1$. The double sum over in-flight pairs $(j, l \in J_j)$ re-indexes as a single sum over $l$, since each $l$ appears in $J_j$ for at most $\tau$ subsequent steps:
\[
\sum_{j=0}^{m-1} \sum_{l \in J_j} \E[(\tilde r_{k_l}^{(l)})^2] \;\leq\; \tau \sum_{l=0}^{m-1} \E[(\tilde r_{k_l}^{(l)})^2].
\]
Hence $E_m \leq E_0 - (1 - 2\rho_{\text{dyn}}\tau)\sum_{j=0}^{m-1}\E[(\tilde r_{k_j}^{(j)})^2]$, and the lower bound (a) converts the stale progress into the stale error. The telescoped bound takes the form
\begin{equation}
\label{eq:async-telescoped}
E_m \;\leq\; E_0 \;-\; \frac{\nu_{\min}^2\,\tilde\nu_\tau\,\lmin}{n}\,\sum_{j=0}^{m-1} \E[\Anorm{\tilde e^{(j)}}^2],
\end{equation}
where $\tilde\nu_\tau = 1 - 2\rho_{\text{dyn}}\tau$ emerges from the AM--GM-induced combination of the two cross-term contributions in exact analogy with ADG's calculation, with $\rho_{\text{dyn}}$ in place of $\rho$ and an additional $\nu_{\min}^2$ multiplier on the leading progress coefficient. The stated condition $\tau < \sqrt n/(2\nu_{\max}\rho_2)$ guarantees $\tilde\nu_\tau > 0$.

\smallskip\noindent\textit{Step 5: Epoch reduction.} ADG's ``error cannot collapse'' argument now applies, with the worst-case contraction rate set by the \emph{upper} side of Lemma~\ref{lem:dynamic-async-1}. The cross terms can have either sign, so they are bounded rather than ignored: reading~\eqref{eq:async-onestep-proof} and~\eqref{eq:async-amgm} as a lower bound and applying (b), (c) and the double-counting of Step~4,
\[
    E_j \;\geq\; E_0 - (1 + 2\rho_{\text{dyn}}\tau)\sum_{i=0}^{j-1}\E[(\tilde r_{k_i}^{(i)})^2]
    \;\geq\; E_0 - \omega\,S_j,
\]
where $\omega := \nu_{\max}^2\,\lmax\,(1 + 2\rho_{\text{dyn}}\tau)/n$ and $S_j := \sum_{i=0}^{j-1}\E[\Anorm{\tilde e^{(i)}}^2]$; the second inequality uses the upper bound of the lemma. Since $\tilde e^{(j)} = e^{(s(j))}$ with $s(j) \leq j$ and $S_j$ is nondecreasing, this gives $\E[\Anorm{\tilde e^{(j)}}^2] \geq E_0 - \omega S_j$, that is, $S_{j+1} \geq E_0 + (1 - \omega)\,S_j$. If $\omega < 1$, induction gives $S_m \geq E_0\,(1 - (1-\omega)^m)/\omega$, which is at least $E_0/(2\omega)$ once $m \geq \ln(2)/\omega$; if $\omega \geq 1$, already $S_1 = E_0 \geq E_0/(2\omega)$. The epoch length in the statement satisfies $m \geq \ln(2)/\omega$, so
\[
    \sum_{j=0}^{m-1}\E[\Anorm{\tilde e^{(j)}}^2] \;\geq\; \frac{E_0}{2\omega} \;=\; \frac{n}{2\,\nu_{\max}^2\,\lmax\,(1 + 2\rho_{\text{dyn}}\tau)}\,E_0.
\]
Substituting into~\eqref{eq:async-telescoped},
\[
    E_m \;\leq\; E_0 - \frac{\nu_{\min}^2\,\tilde\nu_\tau\,\lmin}{n} \cdot \frac{E_0}{2\omega}
    \;=\; \left(1 - \frac{\nu_{\min}^2}{\nu_{\max}^2}\cdot\frac{\tilde\nu_\tau}{2\kappa\,(1 + 2\rho_{\text{dyn}}\tau)}\right) E_0,
\]
where the final equality uses $\lmin/\lmax = 1/\kappa$.
\end{proof}

Both $\nu_{\min}$ and $\nu_{\max}$ appear in the theorem. $\nu_{\min}$
bounds progress from below (imagining the flattest residual) and
multiplies the contraction rate. $\nu_{\max}$ bounds staleness from
above (imagining the spikiest residual) and appears inside
$\rho_{\text{dyn}}$. This pessimism is genuine: when the residual is
spiky the analysis simultaneously (a) underestimates progress using the
flat-residual bound and (b) overestimates staleness using the spiky-residual
bound. See Section~\ref{sec:limitations} for a discussion of this and
other sharpening targets.

The bound in Theorem~\ref{thm:async-power-weighted} differs from the
uniform ADG result in two complementary ways. The leading progress
coefficient $\lmin/n$ is amplified by $\nu_{\min}^2$, an improvement
whenever the residual is non-uniform; the cross-term parameter
$\rho_{\text{dyn}}$ scales as $\nu_{\max}/\sqrt n$ rather than the
uniform $1/n$, a degradation that does not vanish even when the
residual is flat. The two effects pull in opposite directions, and the
ratio $\nu_{\max}/\nu_{\min}$ controls which one dominates. When the
residual is strongly non-uniform along the trajectory (as
Section~\ref{sec:experiments:results} demonstrates is the case for the
problem classes of interest) the progress amplification can
compensate for the looser cross-term, and the bound provides meaningful
information beyond what the uniform ADG result gives.
Table~\ref{tab:adg-comparison} summarizes the comparison.


\begin{table}[ht]
\centering
\small
\renewcommand{\arraystretch}{1.3}
\begin{tabular}{@{}lcc@{}}
\toprule
& \textbf{ADG (uniform)} & \textbf{Dynamic ($\ell=2$)} \\
\midrule
Progress coefficient         & $\lmin/n$                   & $\nu_{\min}^2 \lmin/n$ \\
Cross-term parameter         & $\rho = \|A\|_\infty/n$     & $\rho_{\text{dyn}} = \nu_{\max}\rho_2/\sqrt n$ \\
Scaling (sparse $A$)         & $O(1/n)$                    & $O(\nu_{\max}/\sqrt n)$ \\
Allowable delay              & $\tau = O(n)$               & $\tau = O(\sqrt n/\nu_{\max})$ \\
Epoch length                 & $\geq 0.693\,n/\lmax$       & $\geq (\ln 2)\,n/(\nu_{\max}^2\lmax)$ \\
Epoch reduction              & $1 - \nu_\tau^{\mathrm{ADG}}/(2\kappa)$ & $1 - (\nu_{\min}^2/\nu_{\max}^2)\tilde\nu_\tau/(2\kappa(1 + 2\rho_{\text{dyn}}\tau))$ \\
\bottomrule
\end{tabular}
\caption{Comparison of the ADG uniform-sampling
bound~\cite{avron2015revisiting} with
Theorem~\ref{thm:async-power-weighted}; $\nu_\tau^{\mathrm{ADG}} = 1 - 2\rho\tau$.}
\label{tab:adg-comparison}
\end{table}

Setting $\nu_{\min} = \nu_{\max} = 1$ (perfectly uniform residual)
recovers the ADG epoch-reduction form $1 - \tilde\nu_\tau/(2\kappa)$
up to the factor $1 + 2\rho_{\text{dyn}}\tau \in [1, 2)$, but leaves $\rho_{\text{dyn}} = \rho_2/\sqrt n = O(1/\sqrt n)$. 
Theorem~\ref{thm:async-power-weighted} and ADG Theorem~2a should
therefore be compared as complementary results for different sampling
strategies rather than as a generalization and its special case.
When $\nu_{\min}^2 \approx \nu_{\max}^2$, the per-epoch
reduction approaches the ADG form, up to the same factor, over an epoch of
length $\sim n/(\nu_{\max}^2\lmax)$, a factor $\nu_{\max}^2$ shorter
than the ADG epoch, so that inside the delay window the guaranteed
reduction is reached in correspondingly fewer iterations.

Theorem~\ref{thm:async-power-weighted} inherits the consistent-reads
assumption from the ADG framework, which is the standard setting that
makes the stale residual $\tilde r^{(j)}$ a well-defined snapshot
amenable to the conditional-expectation argument used in the proof.
The hardware experiments in Section~\ref{sec:experiments:results} reveal
that this assumption is empirically the \emph{harder} case rather than
the easier one for power-weighted sampling at high concurrency:
consistent reads destabilize the iteration on all three test problems,
while inconsistent reads diverged only in a minority of runs on one
problem, and the underlying mechanism is explored in
Section~\ref{sec:experiments:results}. 
The theorem above should be read as proving a convergence rate
under an idealized read model rather than as recommending consistent
reads for deployment; a rigorous inconsistent-read analog that captures
the feedback-damping effect is the leading direction of follow-up work.

\subsection{Step-Size Control for Large Delays}
\label{sec:stepsize}

The condition $2\rho_{\text{dyn}}\tau < 1$ in Theorem~\ref{thm:async-power-weighted}
can fail if either $\tau$ or $\nu_{\max}$ is large, and both failures
occur in the hardware experiments of Section~\ref{sec:experiments}: the
measured lags exceed the window at the higher thread counts, and a
point-source start has $\nu_{\max}^2 = n$. The apparent sharpness is an artifact of fixing the
step size at $\beta = 1$ rather than a feature of the iteration itself:
under-relaxation ($\beta < 1$ in~\eqref{eq:async-update}) recovers
a convergence guarantee for \emph{any} finite $\tau$ and $\nu_{\max}$, at a rate that
degrades continuously rather than collapsing at a boundary. The condition
demarcates the regime where the full Jacobi step is feasible, not the
regime where convergence is possible at all. This is the dynamic analog
of ADG Theorem~3.

To see this, redo the unrolled recursion of
Theorem~\ref{thm:async-power-weighted} with $\beta$ in place of $1$. The
progress term (from expanding $\Anorm{x^{(j+1)} - x^*}^2$) picks up a
factor $\beta(2-\beta)$, while the cross-coupling damage between stale
and current updates scales as $\beta^2$ since each commit is
$\beta$-scaled. Progress is therefore roughly linear in $\beta$ while
damage is purely quadratic; the asymmetry that makes shrinking
$\beta$ a net win. Repeating the AM--GM decoupling and double-counting
argument with these factors gives the effective convergence parameter
\[
\tilde\nu_\tau(\beta) = 2\beta - \beta^2 - 2\rho_{\text{dyn}}\tau\beta^2
= 2\beta - \beta^2(1 + 2\rho_{\text{dyn}}\tau).
\]
Convergence requires $\tilde\nu_\tau(\beta) > 0$, which rearranges to
$\beta < 2/(1 + 2\rho_{\text{dyn}}\tau)$. This window contains
$\beta = 1$ precisely when $2\rho_{\text{dyn}}\tau < 1$, so
Theorem~\ref{thm:async-power-weighted}'s condition is exactly the
boundary at which the full Jacobi step ceases to be feasible.
Maximizing $\tilde\nu_\tau(\beta)$ by differentiating gives
$\tilde\beta = 1/(1 + 2\rho_{\text{dyn}}\tau)$, which lies in the
window for every $\rho_{\text{dyn}}\tau \geq 0$ and on substitution
yields
\[
\tilde\nu_\tau(\tilde\beta) = \frac{2}{1 + 2\rho_{\text{dyn}}\tau}
- \frac{1 + 2\rho_{\text{dyn}}\tau}{(1 + 2\rho_{\text{dyn}}\tau)^2}
= \frac{1}{1 + 2\rho_{\text{dyn}}\tau}.
\]
The resulting rate degrades smoothly with $\rho_{\text{dyn}}\tau$: at
$\rho_{\text{dyn}}\tau = 0$ we recover the synchronous full-step rate;
as $\rho_{\text{dyn}}\tau \to \infty$ the rate vanishes but never
becomes negative.

\begin{theorem}[Asynchronous power-weighted Jacobi with step-size control]
\label{thm:async-power-weighted-step}
Consider the asynchronous iteration~\eqref{eq:async-update} with step-size $\beta \in (0, 1]$, unit diagonal $A$, power-weighted sampling ($\ell = 2$), consistent reads, and bounded delay $\tau$ with read times fixed in advance. Under Assumption~\ref{assump:global-ipr}, for any $\tau < \infty$ and any $\nu_{\max} < \infty$, the choice
$$
\tilde\beta = \frac{1}{1 + 2\rho_{\text{dyn}}\tau}
$$
yields the epoch reduction
$$
E_m \leq \left(1 - \frac{\nu_{\min}^2}{\nu_{\max}^2}\cdot\frac{1}{2\kappa\,(3 - 2\tilde\beta)}\right) E_0
$$
for every epoch of length $m \geq (\ln 2)\,n/\big(\nu_{\max}^2\lmax\,\tilde\beta\,(3 - 2\tilde\beta)\big)$ starting from the initial iterate.
\end{theorem}

\begin{proof}
The argument tracks $\beta$-dependence through each step of the proof of Theorem~\ref{thm:async-power-weighted}. The conditional bounds from Step~3 of that proof (Lemma~\ref{lem:dynamic-async-1} and the cross-term inequalities) are statements about expectations over $k_j$ at the stale state and are independent of $\beta$. The $\beta$-dependence enters only through the update~\eqref{eq:async-update} itself. See Appendix~\ref{sect:detailed-proofs} for full details.
\end{proof}

The extension goes through because $\beta$ enters multiplicatively everywhere the component $k_j$ enters. Every term in the unrolled recursion picks up either $\beta$ (from the update itself) or $\beta^2$ (from a cross-coupling), and the AM--GM and double-counting arguments care only about the ratio of these scalings, not about the specific cross-term bound. Any cross-term bound (including the looser dynamic one $\rho_{\text{dyn}}$) slots into the same step-size machinery unchanged.

Since $1 \leq 3 - 2\tilde\beta < 3$, the guaranteed reduction per epoch stays within a factor of three of its value at $\tau = 0$ for every delay, while the epoch lengthens in proportion to $1/\tilde\beta = 1 + 2\rho_{\text{dyn}}\tau$: large delays slow the guaranteed rate continuously instead of voiding the guarantee. With the trivial envelope $\nu_{\min}^2 = 1$, $\nu_{\max}^2 = n$ the theorem needs no assumption on the IPR; then $\rho_{\text{dyn}} = \rho_2$ and $\tilde\beta = 1/(1 + 2\rho_2\tau)$, which is about $1/(2T)$ when the delay is of the order of the thread count $T$. Section~\ref{sec:experiments:results} tests fixed under-relaxation and an adaptive rule motivated by the theorem, $\beta_j = 1/(1 + 2\nu(\tilde r^{(j)})\rho_2\bar\tau/\sqrt n)$ with $\bar\tau = T - 1$, which substitutes the IPR of the current stale residual for $\nu_{\max}^2$ and a nominal delay for $\tau$ (Table~\ref{tab:beta}). The adaptive rule is a heuristic: it is not covered by Theorem~\ref{thm:async-power-weighted-step}, which assumes a fixed step.

\subsection{Limitations}
\label{sec:limitations}

While the present analysis establishes a rigorous asynchronous convergence guarantee for residual-weighted sampling, the theoretical bounds remain conservative in several respects. We identify three specific areas where the current framework invites future refinement.

First, the global IPR bound, with read times fixed in advance, sidesteps the direction--delay coupling (Obstacle~1) to provide a worst-case guarantee. However, the IPR trajectory $\nu^2(\tilde r^{(j)})$ is a physical observable whose evolution is demonstrably smooth in practice. Tracking this evolution directly could replace the global assumption with a dynamic, step-wise analysis.

Second, the Cauchy--Schwarz step produces $\nu(\tilde r)/\sqrt n$, losing a factor of $\sqrt n$ relative to ADG. However, because the sampling probabilities $p_r = (\tilde r_r^{(j)})^2/\|\tilde r^{(j)}\|_2^2$ must sum to 1, they inherently form a probability vector with an $\ell_1$ norm of exactly 1. This property naturally invites the use of H\"older's inequality. Bounding the expected matrix entry via $\|p\|_1 \|A_{l,:}\|_\infty$ rather than Cauchy--Schwarz gives the alternative bound $\max_r |A_{lr}| = \|A_{l,:}\|_\infty$, which for sparse matrices with bounded entries is $O(1)$ in $n$ (worse than the ADG $O(1/n)$ but with no $\nu_{\max}$ factor), whereas the pairing $\|p\|_\infty \|A_{l,:}\|_1 = (\|\tilde r^{(j)}\|_\infty^2/\|\tilde r^{(j)}\|_2^2)\,\|A_{l,:}\|_1$ is $O(1/n)$ for a flat residual, matching ADG. The right answer is likely a convex combination of Cauchy--Schwarz and H\"older keyed to the residual's R\'enyi-2 entropy, recovering something between $O(1/n)$ and $O(1/\sqrt n)$ depending on concentration.

Third, using $\nu_{\min}$ on progress and $\nu_{\max}$ on damage is more conservative than tracking $\nu^2(\tilde r^{(j)})$ per step. In practice, spiky residuals cause large progress boosts \emph{at the same moment} that they cause large damage amplifications; the expected product may behave better than the product of expectations. A step-wise analysis could replace the ratio $\nu_{\max}/\nu_{\min}$ with something closer to 1.

Finally, both theorems are stated for $\ell = 2$ and for a single epoch measured from the initial iterate, as in Proposition~\ref{thm:avron-async}. Other exponents, and the chaining of epochs (which requires accounting for the updates in flight at an epoch boundary), are not treated here.

\section{Numerical Experiments}
\label{sec:experiments}

The experiments in this section serve two purposes. Synchronous runs on a
small suite of structured problems and on eight Matrix Market matrices
test the empirical content of the synchronous theory: whether the IPR
$\nu^2(r) = n\,\|r\|_4^4/\|r\|_2^4$ stays bounded away from its uniform
floor along the trajectory generated by power-weighted sampling, and
whether the resulting sharpened contraction tracks the realized error
decay. A shared-memory OpenMP implementation then examines asynchronous
behavior, motivated by the theory, under emergent rather than synthetic
delays, producing, as a
function of thread count, a natural distribution of completion-count lags
(the measured proxy for the delay $\tau$, defined below) from which we
measure both the realized IPR trajectory and the realized lag statistics
directly. No artificial delays are injected, and a separate
straggler-resilience study, a distinct question about behavior outside the
analytical regime, is left to future work.

We use three SPD test problems based on two matrix-free operators, listed
in Table~\ref{tab:problems}; all three run with the same solver and
instrumentation, differing only in the apply routine, the diagonal, and the
right-hand side. Operators are kept matrix-free so that the per-update work
is the same at every thread count. We do not include a CSR or CSC implementation. The stencil apply touches
at most five entries per update and costs $O(1)$, whereas the residual
scan of the sampler costs $O(n)$ (see below), so an explicit sparse format
would change only the memory traffic of the apply; such a comparison
becomes informative once the sampling cost is amortized.

The \texttt{laplace} problem uses the standard five-point Laplacian
stencil on a uniform $N \times N$ grid with Dirichlet boundary conditions
and the smooth right-hand side $b_k = \sin(\pi x_k)\sin(\pi y_k)$. This
is the eigenfunction-like case in which the residual is dominated by
low-frequency components throughout the run, producing long IPR
trajectories with $\nu^2$ modestly above the uniform floor for many
sweeps; it is the case with smooth forcing and no initial spatial
localization, included to probe the IPR on a residual that has no
localized concentration.

The \texttt{poisson} problem uses the same operator with a localized
point source $b = 100 \cdot e_{N^2/2 + N/2}$ at the grid center. This
is the localized-forcing case where the residual stays spatially
concentrated around the source for many iterations and $\nu^2(r)$
remains elevated well above its uniform floor, giving the sharpened
bound its largest empirical margin over the baseline.

The \texttt{fem} problem uses the consistent mass matrix for
piecewise-linear finite elements on a uniform partition of $[0,1]$
with $n$ nodes and spacing $h = 1/(n-1)$, with a random right-hand
side $b \sim \mathcal{N}(0, I_n)$. The off-diagonals have the opposite
sign convention from the Laplacian ($A_{k,k\pm 1} = h/6 > 0$), the
matrix is well-conditioned, and Jacobi converges rapidly. It serves as a counterpoint to the Laplacian problems: its sustained IPR
is comparable to that of \texttt{laplace} (Table~\ref{tab:ipr-summary}),
but convergence is so fast that the comparison of bound and realized
error is dominated by the start-up phase.

\begin{table}[ht]
\centering
\setlength{\tabcolsep}{3pt}
\begin{tabular}{lccccc}
\hline
Problem        & Operator        & Dim $n$ & $A_{kk}$        & Off-diag         & RHS \\
\hline
\texttt{laplace} & 2D Laplacian   & $N^2$         & $4$             & $-1$             & $\sin(\pi x)\sin(\pi y)$ \\
\texttt{poisson} & 2D Laplacian   & $N^2$         & $4$             & $-1$             & point source (center) \\
\texttt{fem}     & 1D FEM mass    & $N$           & $2h/3$ or $h/3$ & $h/6$            & $\mathcal{N}(0, I_n)$ \\
\hline
\end{tabular}
\caption{Test problems and their matrix-free operator definitions.}
\label{tab:problems}
\end{table}

We report results for $N = 128$ on the 2D problems (so $n = 16{,}384$) and
$n = 8{,}192$ on the \texttt{fem} problem; all runs start from $x^{(0)} = 0$,
so the initial residual is $b$. These sizes are large enough to be well
inside the sparse-with-bounded-row regime of the ADG analysis and small
enough that several hundred sweeps complete in minutes on a single node.
For problems whose residual is initially concentrated (\texttt{poisson} is
the extreme case, with $\nu^2 = n$ at iteration $0$), the IPR decays
toward its sustained level over many sweeps rather than settling quickly,
so unless otherwise noted the IPR statistics reported below are taken
from sweeps $100$--$200$ of $200$-sweep hardware runs. The IPR is still
drifting slowly over that window (on \texttt{poisson} from $6.8$--$7.0$
at sweep $100$ to $4.8$--$4.9$ at sweep $200$, depending on the thread
count), so the window statistics are sustained values rather than
limits.

Hardware runs were carried out on a single $96$-core node (two $48$-core
Intel Xeon Platinum 8468 processors) of the ACES cluster at Texas A\&M
University with \texttt{OMP\_PROC\_BIND=spread}, \texttt{OMP\_PLACES=cores},
and exclusive node allocation; the $T = 128$ runs oversubscribe the node.
The shared-memory implementation is a single-source C++17 OpenMP
program. Each matrix-free operator exposes a diagonal lookup $A_{kk}$,
an initial-residual routine, and an atomic stencil update applied via
\texttt{\#pragma omp atomic update}; the shared state is a single
iterate $x$ and a single residual $r$, both updated atomically. Each
thread runs an independent loop in which it records a dispatch step
(a relaxed atomic load of a global commit counter), reads the residual,
samples a component $k$ by inverse-CDF sampling on the $|r|^\ell$-weighted
distribution, computes the update value, and commits via atomic
fetch-and-add on the same counter; the \emph{completion-count lag} of
that commit is the commit step minus the dispatch step, i.e., the number
of commits by other threads during the thread's own gather, sampling,
and update. In consistent-reads (buffered) mode the gather occurs at the
start of this interval, so the lag is an upper bound on the age (in commits) of every component of the buffer 
and is within the copy duration of the age of
the oldest; in inconsistent-reads (refreshed) mode it is the age of the
sampling distribution, while the committed value has age close to zero. We use
the lag as the measured proxy for the delay parameter $\tau$ of
Section~\ref{sec:async}, which bounds the age of a coherent snapshot;
the comparisons below are meaningful at the level of the ratios
reported, but they are not a test of the parameter itself, which would
require a controlled read model. At the end of each run the maintained
residual was checked against $b - Ax$ recomputed from the final iterate;
in every non-divergent run (in the sense defined below),
$\|r_{\text{maintained}} - (b - Ax)\|_2/\|b - Ax\|_2$ was below
$2\times10^{-8}$ in the runs of Table~\ref{tab:beta}, where the largest
values occur in \texttt{fem} runs that reached the tolerance and the
denominator is smallest, and below $4\times10^{-10}$ (median
$3\times10^{-12}$) in all other runs. All updates use the raw
operators with $\delta = \beta\,r_k/A_{kk}$, $\beta = 1$ unless stated
otherwise. This is the Jacobi step for the symmetrically scaled
system $\hat A = D^{-1/2}AD^{-1/2}$ (unit diagonal, as assumed in
Section~\ref{sec:async}, and the system in which $\rho_2$ is evaluated)
with residual $\hat r = D^{-1/2}r$, and $\|e\|_A = \|\hat e\|_{\hat A}$.
Sampling proportional to $r_k^2 = A_{kk}\hat r_k^2$ coincides with the
theorem's rule $\propto \hat r_k^2$ when the diagonal is constant, which
holds for \texttt{laplace} and \texttt{poisson} and for \texttt{fem}
except at the two boundary nodes; the IPR is computed on the residual of
the unscaled system and coincides with the IPR of $\hat r$ under the same
condition.
The residual norm and IPR logged during a run are computed by thread~0
from the shared residual while the other threads continue to update it,
and are therefore mixed-time estimates; the final values are recomputed
after all threads have joined. Each run ends in one of three states,
which we distinguish throughout: the relative residual reached the
tolerance $10^{-6}$ (the \texttt{fem} runs: within $8.5$--$11$ sweeps under
the power-weighted rules and $29$ under cyclic or uniform sampling); the
$200$-sweep work cap was reached with a finite, decreasing residual (the
\texttt{poisson} and \texttt{laplace} runs, at about $10^{-2}$ on
\texttt{poisson}); or the iteration blew up. A run is counted in the third state,
\emph{diverged}, if its residual overflowed to a non-finite value (the
version of the solver used for Table~\ref{tab:beta} stops at the first
non-finite value it encounters; the earlier version runs on to the cap) or
if it reached the cap with a finite residual more than $10^{47}$ times its
initial norm, which happened in five runs. Runs in the first two states
are \emph{non-divergent}; their relative residual at the end is below $2$
in every case. The tables count non-divergent runs and state their end
state. The IPR
of a diverged run is undefined once the residual is non-finite, and no
IPR statistic in this section includes such runs.

\paragraph{Per-update cost.}
Inverse-CDF sampling over the $|r|^\ell$-weighted distribution costs $O(n)$ per
update, against $O(1)$ for uniform or cyclic selection on a stencil; the
implementation is deliberately the simplest correct one and makes no
attempt to amortize this cost, so all convergence comparisons below are
made per update (per sweep), and the wall-clock figures in
Table~\ref{tab:hw-scaling} include the $O(n)$ sampler. At $T = 96$ on
\texttt{poisson} the measured cost is $28\,\mu$s per update per thread for
power-weighted sampling against $10\,\mu$s for uniform sampling (both
figures include the shared commit counter used for lag measurement, which
dominates the uniform figure at this thread count), and power weighting
reaches a relative residual of $0.05$ after $1.1$ sweeps ($0.01$\,s) where
uniform sampling needs $129$ sweeps ($0.23$\,s); at $T = 1$ the $21\,\mu$s
per update implied by Table~\ref{tab:hw-scaling} is almost entirely the
sampler. A sweep-count advantage is therefore necessary but not sufficient
for a wall-clock advantage: with per-update costs $c_u$ (uniform) and
$c_u + c_s$ (power-weighted), power weighting wins in time only when the
sweep ratio exceeds $1 + c_s/c_u$, about $2.7$ at $T = 96$ and, by our
estimate, two orders of magnitude at $T = 1$ (where no uniform run was
made) for the unamortized sampler. Bringing $c_s$ down is
an implementation question rather than an analytical one: a segment tree
over $|r_j|^\ell$ maintained under stencil updates gives $O(\log n)$
sampling, and, for $\ell = 2$, rejection sampling against an upper bound
$M \ge \max_j r_j^2$ costs $nM/\|r\|_2^2$ expected trials per update;
with the exact maximum this is $n\,\|r\|_\infty^2/\|r\|_2^2$, which is
$O(1)$ when the largest squared component is within a constant factor of
the mean and at most $\nu\sqrt n$ in general (by
$\|r\|_\infty^4 \le \|r\|_4^4$); $M$ must remain an upper bound as the
residual changes for the sampling law to be unchanged. Both sample from a snapshot; in
the model of Section~\ref{sec:async}, a coherent snapshot refreshed every
$K$ updates is consistent-reads execution with delay at most $K$ plus the
in-flight window, so the analysis applies to such samplers at the model
level, though a shared-memory realization would need its own assessment.
We make no time-to-solution claim in this paper.
 
A runtime flag selects between two read modes, which we name after the
ADG settings they approximate. In \textbf{consistent-reads} mode
(buffered-value reuse), each iteration begins with a bulk
\texttt{memcpy} of the shared residual into a thread-local buffer, and
both the CDF construction and the value $r_k$ used for the update come
from that buffer. The copy is not an atomic snapshot: it runs
concurrently with the other threads' atomic stencil updates, so its
components may differ in age by up to the duration of the copy (a few
tens of commits at $T = 96$), and the implementation realizes ADG's
coherent snapshot only up to that window. Formally, these plain loads
racing with atomic stores are a data race under the OpenMP and C++ memory
models, so the consistent-reads runs that use the bulk copy are, strictly,
executions of a program with undefined behavior; the logging thread's plain
loads of the per-thread lag counters are a race in every run, and so are
its loads of the residual in every run made with the original version of
the solver: the runs behind Figures~\ref{fig:hw_convergence}--\ref{fig:hw_scaling}
and Tables~\ref{tab:hw-scaling} and~\ref{tab:consistency-sweep}, the
warm-up scan, and the $\beta = 1$ reference columns in the lower part of
Table~\ref{tab:beta} (the other runs of Table~\ref{tab:beta} use atomic
reads there). We report these runs because the logged quantities
are diagnostics that do not feed the iteration; because on the x86-64
target an aligned eight-byte load is a single instruction, so that in
practice no component is torn, and the end-of-run residual check above
passed in every non-divergent run; and because
the variant in which the buffer is filled by componentwise atomic reads
and its value is still reused, which removes the race from the gather,
diverges in every run as the bulk copy does
(Section~\ref{sec:experiments:results}), so the conclusions about the two
modes do not rest on the racy copy. A strictly conforming implementation
would nonetheless use that atomic gather in both modes. In
\textbf{inconsistent-reads} mode (refreshed-value update), the sampling
distribution is built instead from $n$ componentwise
\texttt{\#pragma omp atomic read}s,
which is also a mixed-time gather, and the value $r_k$ used for the
update is read \emph{a second time}, atomically, after sampling. That second read is the difference between the modes that the
common-gather experiment below isolates: both
sample from a mixed-time residual, but one commits a buffered value that
in-flight updates may have superseded while the other commits the
current value. We attribute the stability contrast reported below to that
difference rather than to snapshot coherence, which neither mode provides,
with the caveat that the two modes also differ in their gather; the variant
in which both share the componentwise gather, reported with the step-size
study below, diverges as the bulk copy does. The
toggle exists to test empirically whether the consistent-reads
assumption underlying our analysis is a meaningful restriction on the
problems in our test suite, following~\cite{avron2015revisiting}.

We report the following measurements: synchronous IPR trajectories and a
comparison of the IPR-sharpened bound (Theorem~\ref{thm:sharp-power})
with realized $A$-norm error, on the three structured problems and on
eight Matrix Market matrices; hardware convergence at $T = 96$ threads
under six sampling rules (cyclic, uniform random, and power-weighted
with $\ell \in \{0.5, 1, 2, 4\}$); a thread-count sweep on \texttt{poisson}
from $T=1$ to $T=128$ that probes the IPR's sensitivity to concurrency
and probes the empirical content of Assumption~\ref{assump:global-ipr};
paired consistent-vs-inconsistent reads at $T = 96$ with matched RNG seeds
that quantify divergence rates; and, for the instability those runs
expose, a study of the step size, the exponent and the gather. 

\subsection{Results}
\label{sec:experiments:results}

\subsubsection*{Synchronous IPR and bound tightness}

Figure~\ref{fig:combined_ipr} reports the IPR trajectory of a single
synchronous ($T = 1$) run of $\ell = 2$ power-weighted sampling on each test
problem, $200$ sweeps on \texttt{laplace} and \texttt{poisson} and to the
tolerance $10^{-6}$ on \texttt{fem}, with the statistics collected in
Table~\ref{tab:ipr-summary}; the \texttt{poisson} run is, update for
update, the $T = 1$ run of Table~\ref{tab:hw-scaling}. The IPR sits above
the uniform floor $\nu^2 = 1$ at every logged step of every run. On
\texttt{poisson} it starts at the single-component value $\nu^2 = n$ and
decays roughly as a power of the sweep count ($\approx 66$, $30$, and $21$
at sweeps $1$, $5$, and $10$; $9.9$ at $50$; $6.8$ at $100$; $4.8$ at
$200$; mean $5.7$ over sweeps $100$--$200$), so the bound's progress
amplification is large in the early phase, where most of the residual
reduction is achieved, and modest later. \texttt{laplace}, which has a
smooth right-hand side and no localized forcing, rises from $\nu^2 = 2.2$
to a plateau near $5.3$ within the first sweep and then drifts down to $3.2$--$3.8$ over sweeps $100$--$200$: random
sampling under power weighting actively drives the residual toward a
non-uniform shape, because each relaxation zeroes the selected component and redistributes
residual to its neighbors while unvisited components keep theirs, so even
a smooth problem develops sustained concentration. The well-conditioned \texttt{fem} problem rises from
$3.0$ to $\approx 5.0$ within the first sweep and settles at $3.9$--$4.3$
until the tolerance is reached after $8.6$ sweeps.

\begin{table}[ht]
\centering
\footnotesize
\setlength{\tabcolsep}{3.5pt}
\begin{tabular}{lccclc}
\toprule
Problem & $n$ & $\nu^2(r^{(0)})$ & $\nu^2$ range, sweeps $\geq 1$ & window & window min / max / mean \\
\midrule
\texttt{laplace} & 16{,}384 & 2.22     & $3.20$--$5.33$ & sweeps $100$--$200$ & $3.20$ / $3.83$ / $3.48$ \\
\texttt{poisson} & 16{,}384 & 16{,}384 & $4.74$--$69.9$ & sweeps $100$--$200$ & $4.74$ / $7.02$ / $5.67$ \\
\texttt{fem}     & 8{,}192  & 3.00     & $3.89$--$4.96$ & sweeps $4.3$--$8.6$ & $3.89$ / $4.32$ / $4.08$ \\
\bottomrule
\end{tabular}
\caption{Inverse participation ratio of the residual along one synchronous
run per problem under $\ell = 2$ power-weighted sampling (seed $42$;
$200$ sweeps, or to the tolerance $10^{-6}$ at $8.6$ sweeps for
\texttt{fem}). The window is sweeps $100$--$200$, or the second half of the
run for \texttt{fem}; the window mean is the ``sustained'' value used in
the text. All statistics are over the logged steps (every $200$ updates).
At the end of the runs $\nu^2 = 3.29$, $4.80$, and $3.94$. The
\texttt{poisson} run coincides with the $T = 1$ run of
Table~\ref{tab:hw-scaling}.}
\label{tab:ipr-summary}
\end{table}

\begin{figure}[ht]
    \centering
    \includegraphics[width=\textwidth]{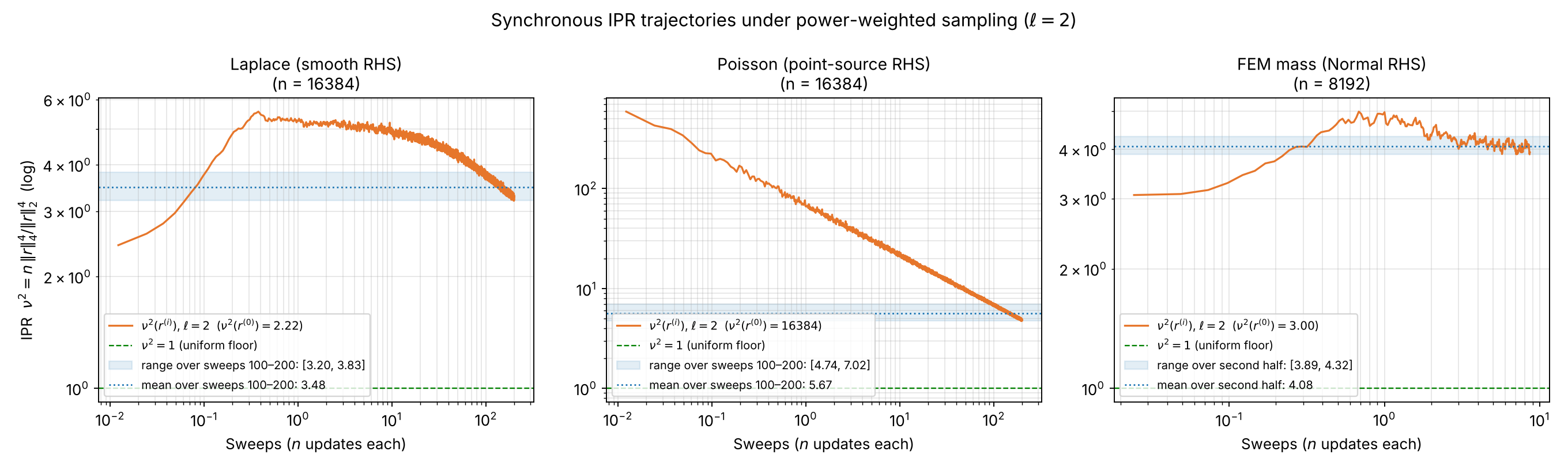}
    \caption{Synchronous IPR trajectories $\nu^2(r^{(i)})$ under
    power-weighted sampling ($\ell = 2$) on the three test problems,
    log--log axes, one run each ($200$ sweeps; \texttt{fem} to tolerance).
    The band and dotted line mark the range and the mean over sweeps
    $100$--$200$ (second half of the run for \texttt{fem}). \texttt{poisson}
    starts at $\nu^2 = n$ (not shown on the logarithmic axis; see
    Table~\ref{tab:ipr-summary}) and decays roughly as a power of the sweep count; \texttt{laplace}
    and \texttt{fem} rise from their initial values to a plateau within the
    first sweep and then drift slowly. The persistent gap above $1$ is the
    empirical content of the trajectory assumption in
    Corollary~\ref{cor:global-rate} and
    Theorem~\ref{thm:async-power-weighted}.}
    \label{fig:combined_ipr}
\end{figure}

Figure~\ref{fig:combined_sync_bounds} compares, on the same runs, the
realized squared $A$-norm error $\Anorm{e^{(i)}}^2/\Anorm{e^{(0)}}^2$
($x^*$ computed by conjugate gradients to a relative residual of
$10^{-14}$; the error is advanced by the exact progress
identity~\eqref{eq:progress-identity} and re-evaluated from $x^{(i)} - x^*$
once per sweep, with relative agreement better than $2\times10^{-11}$)
with the products of the per-step factors of the uniform bound of
Proposition~\ref{thm:uniform}, $\prod_{j<i}(1 - \lmin/(nA_{\max}))$, and
of the IPR-sharpened bound of Theorem~\ref{thm:sharp-power} evaluated
along the realized trajectory, $\prod_{j<i}(1 - \nu_j^2\lmin/(nA_{\max}))$, the normalizer of
Corollary~\ref{cor:geometric-mean}. All three are squared quantities. The
diagonal-weighted bound of Proposition~\ref{thm:leventhal-lewis}
coincides with the uniform one on the constant-diagonal Laplacian
problems and differs from it by the factor $n/(n-1)$ in the exponent on
\texttt{fem}, and is not drawn. The IPR curve is a diagnostic of how much
of the realized decay the one-step bound accounts for, not an a priori
bound, and it lies between the uniform bound and the realized error on
every problem: after $200$ sweeps, $0.79$ against a realized $0.69$ on
\texttt{laplace} and $0.56$ against $0.093$ on \texttt{poisson}, with the
uniform bound at $0.94$ on both; at the tolerance on \texttt{fem},
$1.2\times10^{-6}$ against $1.2\times10^{-12}$, with the uniform bound at
$0.039$. The gap is smallest on \texttt{laplace}, whose residual has no
localized concentration, and, at the common horizon of $200$ sweeps,
largest on \texttt{poisson}, where the
realized error falls by a factor of three within the first sweep while
the IPR curve, which sees the transient only through $\nu_i^2$ and
otherwise uses $\lmin$, moves by $4\%$: the one-step bound measures the
residual against the smallest eigenvalue, whereas the point-source
residual lives in high-frequency modes with much larger Rayleigh
quotients. Tracking the residual's support in the eigenbasis of $A$
alongside its concentration is the natural avenue for further sharpening.

\begin{figure}[ht]
    \centering
    \includegraphics[width=\textwidth]{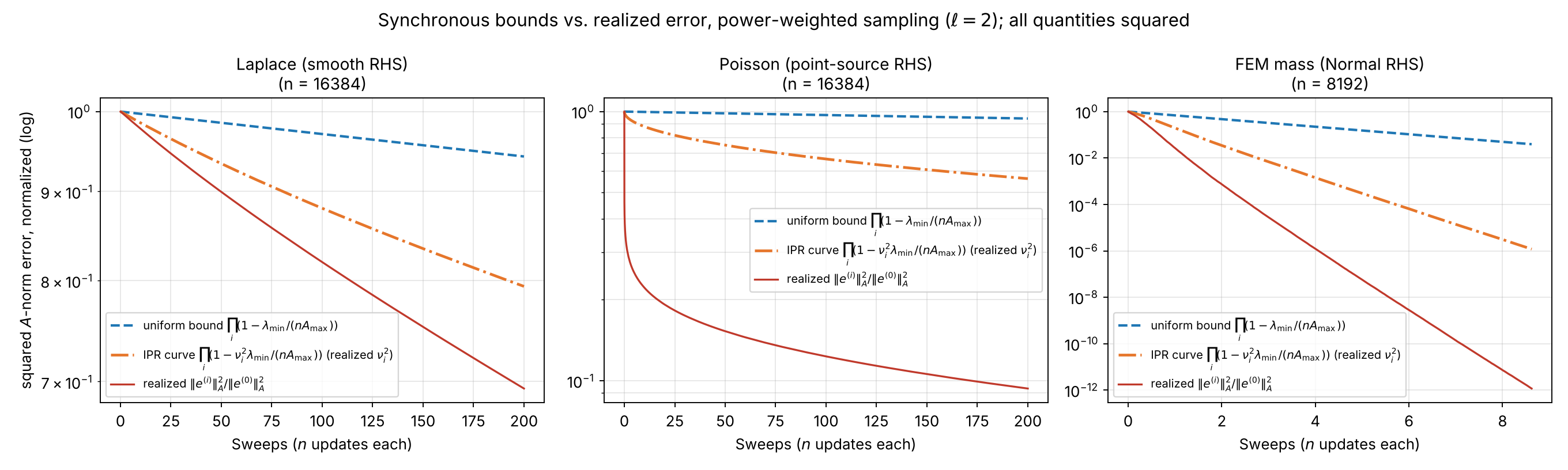}
    \caption{Realized squared $A$-norm error (solid) against the products
    of the per-step factors of the uniform bound (dashed,
    Proposition~\ref{thm:uniform}) and of the IPR-sharpened bound evaluated
    along the realized trajectory (dash-dotted, Theorem~\ref{thm:sharp-power}
    and Corollary~\ref{cor:geometric-mean}) on the synchronous runs of
    Figure~\ref{fig:combined_ipr}; the realized error is squared and
    normalized by its initial value, and both products start at $1$. The
    IPR curve lies between the uniform bound and the realized error on
    every problem, with the gap to the realized error smallest on
    \texttt{laplace} and, at the common $200$-sweep horizon, largest on
    \texttt{poisson} (on \texttt{fem}, at its tolerance, the gap is six
    orders of magnitude).}
    \label{fig:combined_sync_bounds}
\end{figure}

\subsubsection*{Matrix Market test matrices}

The three test problems have constant diagonal up to two entries, which is
the case in which the implemented sampling rule $\propto r_k^2$ coincides
with the rule $\propto \hat r_k^2$ of the unit-diagonal analysis and in
which Remark~\ref{rem:exact-constant-diag} applies. To check that the
residual concentration the sharpened bounds rely on is not an artifact of
that structure, we ran the same synchronous iteration on eight standard
SPD test matrices from the NIST Matrix Market~\cite{boisvert1997matrix},
seven from the Harwell--Boeing collection and \texttt{s1rmq4m1} from the
Cylshell set (they are also available in the HB and Cylshell groups of
SuiteSparse~\cite{davis2011university}), listed in
Table~\ref{tab:hb}: a nine-point Laplacian, a power-network admittance
matrix, a biharmonic plate and five structural stiffness matrices, with
$n$ from $900$ to $11{,}948$, between $4\times10^{3}$ and $2.9\times10^{5}$
nonzeros, and diagonal ratios $\max_k A_{kk}/\min_k A_{kk}$ from $1$ to
$9\times10^{10}$. Each matrix is symmetrically scaled to unit
diagonal, $\hat A = D^{-1/2}AD^{-1/2}$, which preserves the coordinate-update
correspondence for a given index sequence but changes the sampling law to
$\propto \hat r_k^2$, the rule of Theorems~\ref{thm:sharp-power}
and~\ref{thm:async-power-weighted};
$b \sim \mathcal N(0, I_n)$ is drawn for the original system and scaled
with it, $x^{(0)} = 0$, $\ell = 2$, and a run ends at
$\|\hat r\|_2/\|\hat b\|_2 \le 10^{-6}$ or after $100$ sweeps (all residual
norms quoted for the scaled runs are those of the scaled system; $x^*$ is
computed by conjugate gradients on the scaled system to a relative
residual of $10^{-14}$; $\lmin$ of $\hat A$ and $A$ by shift-invert Lanczos; $\lmax$, used only
for the reported $\kappa$, by power iteration). The table reports the
IPR of the initial residual, its mean over the second half of the run, the
effective amplification
$\bar a = \sum_i \ln(1 - \nu_i^2\lmin/n)\big/\sum_i \ln(1 - \lmin/n)$,
i.e., the factor by which the realized IPR curve outpaces the uniform bound
in the exponent, and the end-of-run values of the realized squared error
and of both products (``IPR'' and ``unif.''); the last two columns repeat the sustained IPR and
the realized error for the unscaled matrix, on which the sampler draws
$\propto r_k^2$.

Table~\ref{tab:hb} shows three things. First, the sustained IPR lies
between $3.0$ (the Laplacian) and $16$ (\texttt{bcsstk18}) on every
matrix, well above the floor $\nu^2 = 1$ at which the sharpened bound
reduces to the uniform one, and the amplification $\bar a$ is $3.2$--$6.2$
on seven matrices and $24$ on \texttt{bcsstk18}. On the scaled systems
the initial residual $\hat b = D^{-1/2}b$ is itself concentrated on the
small-diagonal rows ($\nu^2(\hat r^{(0)}) = 261$ on \texttt{bcsstk16} and
$1{,}736$ on \texttt{bcsstk15}), which the sampler visits first: within the
first sweep the residual falls by four orders of magnitude on
\texttt{bcsstk16} and by $80\times$ on \texttt{bcsstk15}, and the IPR
settles near its sustained value. Second, the ordering uniform bound
$\geq$ IPR curve $\geq$ realized error holds at every logged step of every
run, and, as on \texttt{poisson} and \texttt{fem}, the second gap is the
larger in ratio.
Both products are worst-case in the spectrum, through $\lmin/n$: on the
five matrices with $\kappa(\hat A) \geq 3\times10^4$ they stay above
$0.82$ over the $100$ sweeps while the realized error falls to between
$1.2\times10^{-5}$ (\texttt{bcsstk18}) and $0.89$ (\texttt{nos3}), whereas on
the three matrices with $\kappa(\hat A)$ in the hundreds or low thousands the
IPR curve is
substantially tighter ($0.083$ against $0.46$ on \texttt{gr\_30\_30},
$0.18$ against $0.67$ on \texttt{bcsstk16}). Third, on the unscaled
matrices with large diagonal ratio the sampler $\propto r_k^2$
concentrates on the large-diagonal rows. Except on \texttt{bcsstk16},
where neither quantity changes appreciably, the sustained IPR is
$1.7$--$11\times$ larger ($25$ on \texttt{1138\_bus}, $64$ on
\texttt{bcsstk15}, $168$ on \texttt{bcsstk18}; these are means of
heavy-tailed sequences, the logged window values on \texttt{bcsstk18}
ranging from $14$ to $8.8\times10^{3}$ with median $54$), and the squared
$A$-norm error $\varepsilon_m$ at the end of the run is larger by
$1.3\times$ on \texttt{bcsstk15}, $1.9\times$ on
\texttt{bcsstk27}, $6\times$ on \texttt{s1rmq4m1} and $22\times$ on
\texttt{bcsstk18}, though not on \texttt{1138\_bus}, where the two
variants end at the same error. Sampling on the scaled residual is the
safer default when the diagonal varies.

\begin{table}[ht]
\centering
\scriptsize
\setlength{\tabcolsep}{2.2pt}
\renewcommand{\arraystretch}{1.08}
\begin{tabular}{@{}lrrr|rrrrrr|rr@{}}
\toprule
 & & diag & & \multicolumn{6}{c|}{unit-diagonal scaling $\hat A$} & \multicolumn{2}{c}{raw $A$} \\
Matrix & $n$ & ratio & $\kappa(\hat A)$ & $\nu^2(\hat r^{(0)})$ & sust.\ $\nu^2$ & $\bar a$ & $\varepsilon_m$ & IPR & unif. & sust.\ $\nu^2$ & $\varepsilon_m$ \\
\midrule
\texttt{gr\_30\_30} & 900 & $1$ & $1.9\mathrm{e}{2}$ & 2.9 & 3.0 & 3.2 & $3.8\mathrm{e}{-4}$ & 0.083 & 0.46 & 3.0 & $3.8\mathrm{e}{-4}$ \\
\texttt{nos3} & 960 & $6$ & $3.3\mathrm{e}{4}$ & 3.2 & 4.3 & 4.3 & $0.89$ & 0.97 & 0.99 & 4.6 & $0.89$ \\
\texttt{1138\_bus} & 1{,}138 & $3\mathrm{e}{4}$ & $4.9\mathrm{e}{5}$ & 21.7 & 5.8 & 5.6 & $0.78$ & 1.00 & 1.00 & 25.0 & $0.78$ \\
\texttt{bcsstk27} & 1{,}224 & $5\mathrm{e}{3}$ & $2.1\mathrm{e}{3}$ & 37.1 & 5.5 & 5.2 & $7.1\mathrm{e}{-4}$ & 0.33 & 0.81 & 9.2 & $1.3\mathrm{e}{-3}$ \\
\texttt{bcsstk15} & 3{,}948 & $5\mathrm{e}{9}$ & $8.2\mathrm{e}{4}$ & 1{,}736 & 6.1 & 6.2 & $8.7\mathrm{e}{-4}$ & 0.97 & 0.99 & 64.2 & $1.2\mathrm{e}{-3}$ \\
\texttt{bcsstk16} & 4{,}884 & $2\mathrm{e}{9}$ & $8.0\mathrm{e}{2}$ & 261 & 4.1 & 4.4 & $7.9\mathrm{e}{-11}$ & 0.18 & 0.67 & 4.1 & $6.8\mathrm{e}{-11}$ \\
\texttt{s1rmq4m1} & 5{,}489 & $2\mathrm{e}{4}$ & $4.3\mathrm{e}{5}$ & 9.5 & 4.0 & 4.5 & $5.7\mathrm{e}{-3}$ & 1.00 & 1.00 & 9.0 & $3.6\mathrm{e}{-2}$ \\
\texttt{bcsstk18} & 11{,}948 & $9\mathrm{e}{10}$ & $5.6\mathrm{e}{4}$ & 41.3 & 16.0 & 24.1 & $1.2\mathrm{e}{-5}$ & 0.82 & 0.99 & 168 & $2.6\mathrm{e}{-4}$ \\
\bottomrule
\end{tabular}
\caption{Synchronous $\ell = 2$ power-weighted Jacobi on eight SPD test
matrices from the Matrix Market (seed $42$, $b \sim \mathcal N(0, I_n)$, $x^{(0)} = 0$, at most
$100$ sweeps or to $\|\hat r\|_2/\|\hat b\|_2 \le 10^{-6}$, which only
\texttt{bcsstk16} reached, after $93$ sweeps). ``diag ratio'' is
$\max_k A_{kk}/\min_k A_{kk}$ of the original matrix; $\kappa$ is that of
the scaled system. Sustained $\nu^2$ is the mean over the second half of
the run; $\bar a$ is the effective amplification of the uniform bound's
exponent by the realized IPR; $\varepsilon_m$, IPR and unif.\ are the
end-of-run values of $\varepsilon_m = \Anorm{e^{(m)}}^2/\Anorm{e^{(0)}}^2$,
$\prod_{i<m}(1 - \nu_i^2\lmin/n)$ and $\prod_{i<m}(1 - \lmin/n)$ for the
scaled system. The raw columns use the same $b$ and the same $E_0$.}
\label{tab:hb}
\end{table}

\subsubsection*{Hardware convergence and IPR stability under concurrency}

Under emergent rather than synthetic delays, power-weighted sampling
translates the IPR advantage into a measurable per-update convergence
improvement over uniform random sampling.
Figure~\ref{fig:hw_convergence} reports residual decay under six sampling
rules at $T = 96$ threads on a single ACES node. On \texttt{poisson} the
power-weighted rules reach a relative residual of $0.05$ within
$1.1$--$1.9$ sweeps, against $39$ sweeps for cyclic and $129$ for uniform
random sampling, and reach $0.02$ by sweeps $32$--$68$, which neither
cyclic nor uniform sampling does within the $200$-sweep cap (final values
$0.0125$--$0.015$ for the power-weighted rules against $0.024$ and
$0.041$). On \texttt{laplace} no rule reaches the tolerance within $200$
sweeps, and the $2$-norm of the residual, which coordinate relaxation
does not decrease monotonically, first rises (to $1.1$--$1.4$ times its
initial value) before decaying; the power-weighted rules lie below
uniform random sampling from the first sweep on (final $1.05$--$1.19$
against $1.33$), but cyclic sampling is the best rule on this problem
(final $0.97$). On \texttt{fem} the power-weighted rules reach the
$10^{-6}$ tolerance in $8.5$--$10.9$ sweeps against $29$ for cyclic or
uniform sampling. The near-coincidence of the $\ell = 2$ and $\ell = 4$
curves suggests diminishing returns from increasing residual bias once
concentration is fully exploited.

\begin{figure}[ht]
    \centering
    \includegraphics[width=0.94\textwidth]{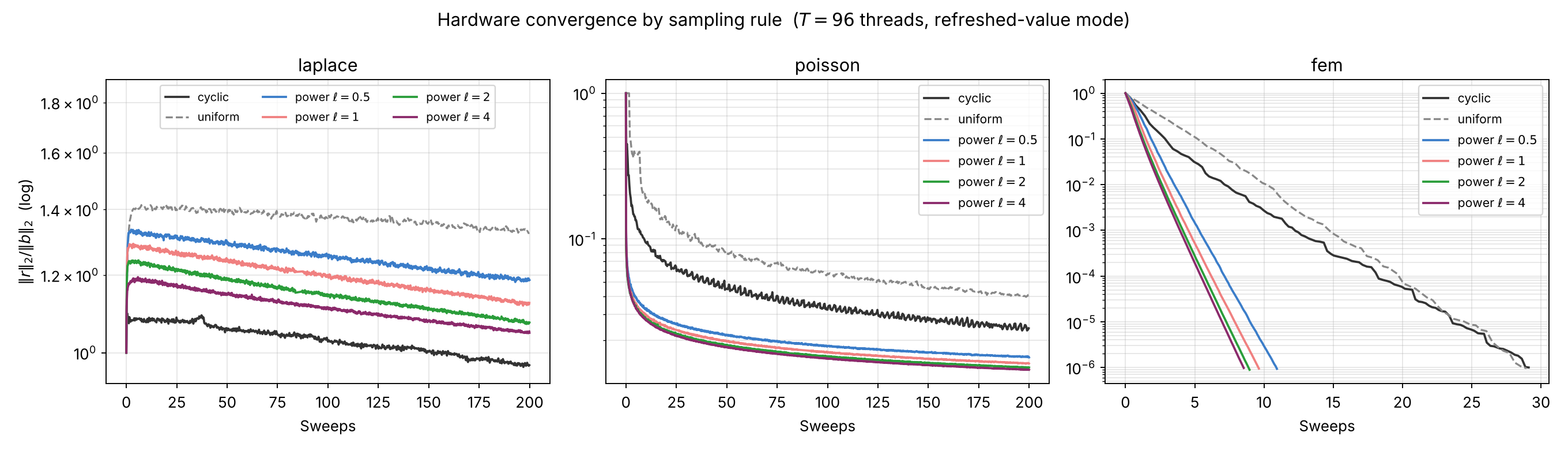}
    \caption{Hardware convergence under six sampling rules at $T=96$ threads,
    inconsistent reads (refreshed-value update). Each curve is a single run
    with the common seed ($42$) of the rule comparison, except the
    $\ell = 2$ curve on \texttt{poisson}, which is the first matched-seed
    run (seed $101$) of Table~\ref{tab:consistency-sweep} because the
    seed-$42$ run of that configuration diverged (consistent with the observed divergence rate
    reported in the text).
    Power-weighted sampling converges faster per update than uniform random
    sampling on all three problems and faster than cyclic sampling on
    \texttt{poisson} and \texttt{fem}; on \texttt{laplace} cyclic sampling
    is the best rule. The $2$-norm residual on \texttt{laplace} rises before
    it decays, as coordinate relaxation does not decrease it monotonically.}
    \label{fig:hw_convergence}
\end{figure}

The hardware IPR trajectories for $\ell = 2$ in Figure~\ref{fig:hw_ipr}
closely follow their synchronous counterparts, indicating that hardware-induced delays
do not alter the residual concentration that the envelope of
Assumption~\ref{assump:global-ipr} describes (the assumption concerns a
specific envelope and is not itself tested by these runs). Uniform and
cyclic sampling produce
noticeably \emph{higher} IPR values on \texttt{poisson} and \texttt{fem}
because they reduce the dominant residual components more slowly and
leave more residual mass concentrated in a small set of components, which
power weighting exploits.

\begin{figure}[ht]
    \centering
    \includegraphics[width=0.94\textwidth]{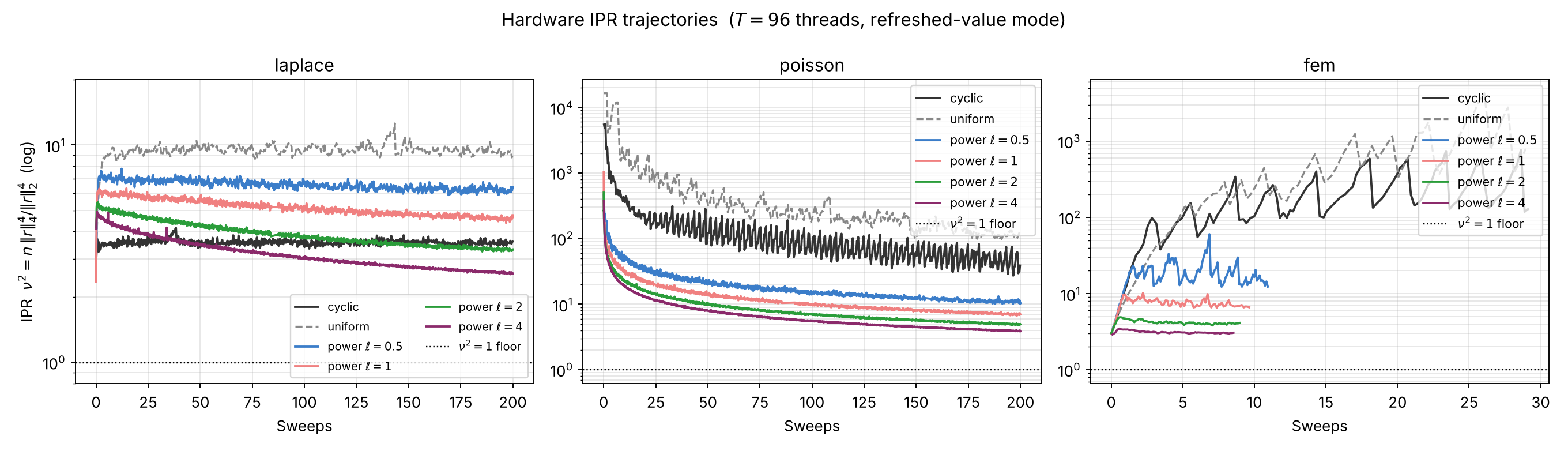}
    \caption{Hardware IPR trajectories under the configuration of
    Figure~\ref{fig:hw_convergence}. The $\ell = 2$ curves track the
    synchronous trajectories of Figure~\ref{fig:combined_ipr}; uniform
    sampling leaves substantially more concentration in the residual on all
    three problems, and cyclic sampling on \texttt{poisson} and
    \texttt{fem}.}
    \label{fig:hw_ipr}
\end{figure}

The thread-count sweep on \texttt{poisson} (Figure~\ref{fig:hw_scaling}
and Table~\ref{tab:hw-scaling}) gives three findings. Wall time scales
near-linearly up to the full node (parallel efficiency $0.77$ at $T = 96$
relative to $T = 1$) and does not improve further at $T = 128$, where the
$96$-core node is oversubscribed. The mean completion-count lag tracks $T - 1$
closely, as it must when every thread is busy and commits at the same
rate, while the maximum lag grows much faster than linearly (from $65$ at
$T = 4$ to about $10^4$ at $T = 64$--$96$); the maximum is a
single-event statistic that records the longest stall of any one thread.
For reference, the sustained IPR gives a certified delay window
$\tau < \sqrt n/(2\nu_{\max}\rho_2)$ of
Theorem~\ref{thm:async-power-weighted} of about $24$ at this problem size
for a flat envelope at the sustained value, about $21$ with the
sweeps-$100$--$200$ extremes $\nu^2 \in [4.8, 7.2]$, and about $12$
commits if the envelope must also cover the post-transient value
$\nu^2 \approx 21$ at sweep $10$. Consequently, the realized mean lags
exceed the certified window from $T = 16$--$32$ onward, and the sweep
probes the iteration well outside the regime the theorem certifies (and,
in any case, under inconsistent rather than consistent reads); the
convergence observed there is evidence that the cross-term bound is
conservative, not a test of the theorem. Most importantly for the theory, the sustained $\nu^2$
(mean over sweeps $100$--$200$) varies by $0.10$ in absolute terms (under $2\%$) across the entire sweep from $T = 1$ to $T = 128$, and
the slow drift within the window (from $6.8$--$7.0$ at sweep $100$ to
$4.8$--$4.9$ at sweep $200$) is nearly the same at every thread count. This is
evidence that residual concentration is insensitive to concurrency in
these tests, not that Assumption~\ref{assump:global-ipr} holds with a 
particular envelope: even under heavy
concurrent execution with realized mean lags in the $10^2$ commit range,
the residual concentration remains essentially identical to the
synchronous baseline, which supports treating the IPR as a property of the
iteration trajectory rather than an artifact of the delay model.

\begin{figure}[ht]
    \centering
    \includegraphics[width=0.94\textwidth]{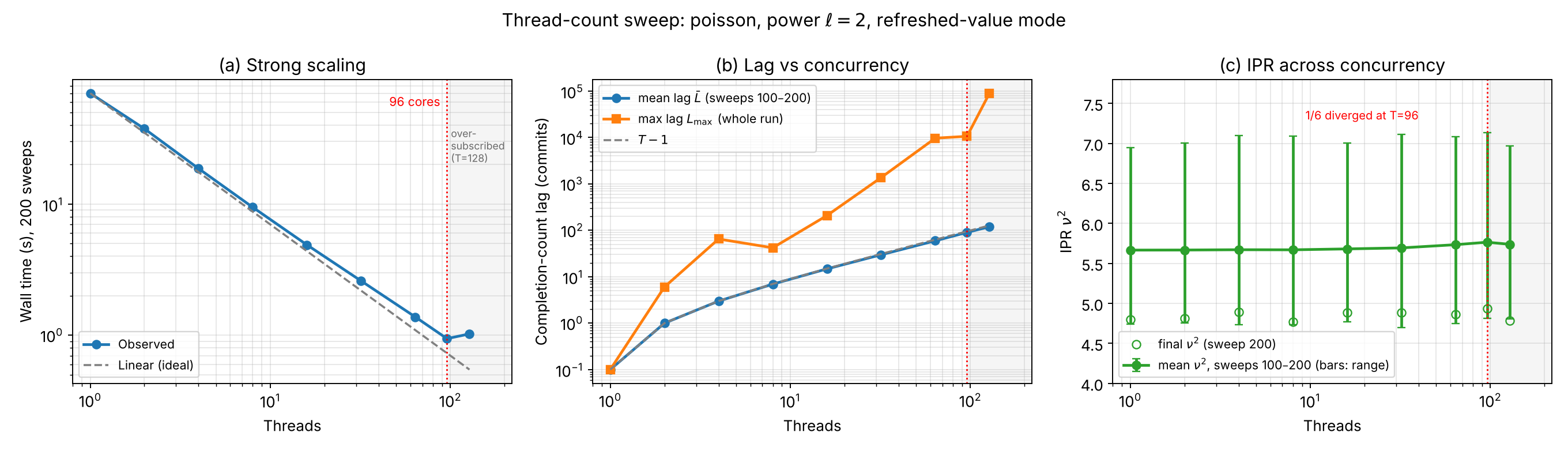}
    \caption{Thread-count sweep for \texttt{poisson}, power-weighted
    ($\ell=2$), inconsistent reads (refreshed-value update), $200$ sweeps
    per run. Points for $T \leq 64$ and $T = 128$ are single runs; the
    $T = 96$ point averages the five non-divergent runs of the first
    matched-seed batch of Table~\ref{tab:consistency-sweep}. The shaded
    region marks oversubscription of the $96$-core node. (a) Wall time
    scales near-linearly up to the full node. (b) The mean completion-count
    lag $\bar L$ over sweeps $100$--$200$ tracks $T - 1$; the maximum lag
    $L_{\max}$ over the whole run grows super-linearly. From
    $T = 16$--$32$ on, $\bar L$ exceeds the conservative delay window of
    Theorem~\ref{thm:async-power-weighted} while the iteration remains
    non-divergent. (c) Mean IPR over sweeps $100$--$200$ (bars: range over
    the window; hollow markers: value at sweep $200$) is essentially
    independent of thread count, supporting the use of the IPR as a
    concurrency-insensitive diagnostic. The divergence of one
    of the six $T = 96$ runs is tallied in Table~\ref{tab:consistency-sweep}.}
    \label{fig:hw_scaling}
\end{figure}

\begin{table}[ht]
\centering
\small
\renewcommand{\arraystretch}{1.1}
\begin{tabular}{r r r r r}
\toprule
$T$ & Wall time (s) & $\bar L$ & $L_{\max}$ & Sustained $\nu^2$ \\
\midrule
  1  &   70.05 &     0.0 &         0 & 5.67 \\
  2  &   37.57 &     1.0 &         6 & 5.67 \\
  4  &   18.66 &     3.0 &        65 & 5.67 \\
  8  &    9.46 &     6.9 &        42 & 5.67 \\
 16  &    4.89 &    14.8 &       208 & 5.68 \\
 32  &    2.60 &    29.8 &    1{,}368 & 5.69 \\
 64  &    1.38 &    59.5 &    9{,}667 & 5.73 \\
 96  &    0.94 &    90.2 &   10{,}565 & 5.76 \\
128  &    1.02 &   119.4 &   89{,}607 & 5.74 \\
\bottomrule
\end{tabular}
\caption{Hardware scaling for \texttt{poisson}, $N=128$, power-weighted
($\ell=2$), inconsistent reads (refreshed-value update), $200$ sweeps per
run. $\bar L$ is the mean completion-count lag (commits by other threads
between a thread's dispatch and its own commit) over sweeps $100$--$200$,
and $L_{\max}$ the maximum lag over the whole run; sustained $\nu^2$ is
the mean IPR over sweeps $100$--$200$. Both window means weight each
logged row by the number of updates since the previous row; this matters
only at $T = 128$, where the logging thread is intermittently descheduled
(the unweighted means there are $114.1$ and $5.66$). The $T = 96$ row averages the five
non-divergent runs of the first matched-seed batch at this thread count
(Table~\ref{tab:consistency-sweep}), with $L_{\max}$ the largest run
maximum; the $T = 128$ row oversubscribes the $96$-core node. The wall
time at $T = 1$ is dominated by the unamortized $O(n)$ sampler (see the
per-update cost discussion above).}
\label{tab:hw-scaling}
\end{table}

\subsubsection*{Consistent vs.\ inconsistent reads}

The original ADG analysis treats consistent reads as the easy case for
the convergence theory and inconsistent reads as a more delicate setting
that requires a separate analysis (\cite[Section~7]{avron2015revisiting}).
Under power-weighted sampling on a concentrated residual, the empirical
ordering on our test problems is reversed. Table~\ref{tab:consistency-sweep}
reports paired runs on all three problems at $T = 96$ with six matched
seeds; on \texttt{poisson}, inconsistent reads diverged in 1 of 6 runs,
while consistent reads diverged in all 6. The outcome of an inconsistent-reads run is not reproducible
from its seed, since the thread interleaving is not controlled by it:
across all $25$ inconsistent-reads runs of this configuration in our
records (the six matched seeds run twice, five further repeats, two
earlier runs, and the six $\beta = 1$ runs of Table~\ref{tab:beta}) $8$
diverged, including one of the two runs of a single seed, while all $31$
consistent-reads runs of the configuration diverged.
We draw no timing conclusion from the paired runs: every consistent-reads
run at baseline $n$ diverged early and then ran on non-finite values to
the $200$-sweep cap, so the per-sweep costs of the two modes are not
comparable from these runs. On \texttt{fem}, inconsistent reads reach the
$10^{-6}$ tolerance within nine sweeps under $\ell = 2$ (and within $30$
sweeps across all sampling rules), whereas the consistent runs overflowed
within two sweeps.

None of this contradicts Theorem~\ref{thm:async-power-weighted}, which
assumes consistent reads with coherent snapshots, a read model the
implementation does not provide, and certifies convergence only inside a
delay window, which is empty here. On \texttt{poisson} the point-source
start has $\nu^2(r^{(0)}) = n$, so any envelope satisfying
Assumption~\ref{assump:global-ipr} on a horizon that includes iteration $0$
has $\nu_{\max} = \sqrt n$, and the certified delay window is
$\tau < 1/(2\rho_2) \approx 0.45$: the theorem certifies nothing for $T \geq 2$ on
this horizon, so the divergence at every tested $T \geq 8$ in
Table~\ref{tab:beta} lies outside its scope rather than contradicting it. A
sufficient condition that fails does not predict divergence; what the
theorem does say is that the collapse of the sampling distribution onto
few components, the quantity that closes its window, is the feature the
pile-on described below exploits. The logs show how the runs are lost: in every consistent-reads
run at $\beta = 1$ across $8 \leq T \leq 96$ in Table~\ref{tab:beta}, the
relative residual exceeds $10$ within the first $0.2$ sweep (within about
$0.1$ sweep in the thread-count sweep), and in six of the twelve
thread-sweep runs at $T = 64$ and $T = 96$, whose logs resolve the first
few hundredths of a sweep, it first dips to $0.2$--$0.3$ of the initial
value; the iteration is lost during the transient itself.
What Section~\ref{sec:async} does not cover at all is the stability of
inconsistent reads on the same runs. The remedy the theory does offer is
the step-size control of Theorem~\ref{thm:async-power-weighted-step},
whose transient prescription $\tilde\beta \approx 1/(2T)$ is damping of the
order the pile-on described next calls for; the step-size study below
finds that fixed steps of that order do remove the divergence, with the
prescription conservative by about an order of magnitude.

The mechanism we conjecture is that the IPR plays a dual role under
power weighting: it amplifies expected progress
(Theorem~\ref{thm:sharp-power}) but also amplifies the collision rate
between concurrent threads. The argument in~\cite{avron2015revisiting}
that inconsistent reads are rare in the sparse reference scenario relies on each update landing on a uniformly random component:
another thread's update hits a fixed read set $\mathcal R$ with probability
at most $|\mathcal R|/n$. Under power weighting with $\ell = 2$ the per-component
update probabilities are $r_k^2 / \|r\|_2^2$ rather than $1/n$, and a
short calculation using $\|r\|_\infty^4 \leq \|r\|_4^4 = \nu^2 \|r\|_2^4 / n$
shows the collision-probability bound becomes
$|\mathcal R| \cdot \sqrt{\nu^2/n}$. The per-iteration collision probability is
therefore inflated by a factor at most $\sqrt{n\nu^2}$ relative to the
uniform case; roughly $305\times$ at the parameters of our \texttt{poisson}
runs. Under independent draws from a common sampling distribution
(collisions were not instrumented in the hardware runs), the expected
number of unordered thread pairs that select the same component in one
sampling round is $\binom{T}{2}\nu^2/n$, i.e., the pairwise
collision rate is inflated by the factor $\nu^2$ relative to uniform
sampling: about $3.5$ on \texttt{laplace} and $5.7$ on \texttt{poisson} in
the sustained regime, and up to $n$ during the point-source transient
(the factor $\sqrt{n\nu^2}$ above is the fixed-read-set worst case).

In the proposed mechanism, under consistent reads such a collision
produces a frozen-value pile-on: several
threads commit corrections computed from the same stale value of one
component before any of the commits is visible to the others. In the
simplest case, $m$ threads committing the same stale value multiply the
residual at that component by $1 - m$, i.e.\ amplify its magnitude by
$m - 1$, which is unstable for $m \geq 3$ (at $m = 2$ the residual merely
changes sign). The \texttt{poisson} initial $\nu^2 \approx n$ then
produces divergence within a fraction of a sweep, and the smaller
\texttt{laplace} and \texttt{fem} rates produce a slower accumulation that
nonetheless reaches divergence, in the six matched-seed runs of
Table~\ref{tab:consistency-sweep}, within $26$ sweeps on \texttt{laplace}
($1.4$--$25.6$ sweeps, the residual never falling below its initial value)
and within two sweeps on \texttt{fem} (after first dropping to
$0.16$--$0.7$ of the initial residual); in other runs of the same
configurations it took up to $4$ sweeps on \texttt{fem} and as long as
$98$ and $191$ sweeps on \texttt{laplace}. Under
inconsistent reads, the second residual read used to set the update value
sees the partial effect of any in-flight commits to the colliding
component and can reduce the resulting overshoot accordingly: a
sign-aware correction that the absolute-value bound above cannot capture.
The same $\nu^2$ that amplifies progress in Theorem~\ref{thm:sharp-power}
also amplifies collision probability ($\nu^2/n$ per pair); under
consistent reads these roles compound destructively, while under
inconsistent reads the feedback mechanism mediates between them.

\begin{table}[ht]
\centering
\small
\renewcommand{\arraystretch}{1.1}
\setlength{\tabcolsep}{4pt}
\begin{tabular}{l r r l c c}
\toprule
Problem & $n$ & $T$ & End state of non-divergent runs & Inconsistent & Consistent \\
\midrule
\multicolumn{6}{l}{\textit{Original sweep (baseline $n$)}}\\
\texttt{poisson} & $16{,}384$ & 96 & cap, $\|r\|_2/\|b\|_2 = 1.3\times10^{-2}$ & 5/6 & 0/6 \\
\texttt{laplace} & $16{,}384$ & 96 & cap, $1.07$ (from a peak of $1.25$) & 6/6 & 0/6 \\
\texttt{fem}     & $8{,}192$  & 96 & tolerance $10^{-6}$, within $9$ sweeps & 6/6 & 0/6 \\
\midrule
\multicolumn{6}{l}{\textit{Scaling test ($4\times$ larger $n$)}}\\
\texttt{poisson} & $65{,}536$ & 96 & cap, $9\times10^{-3}$ & 6/6 & 0/6 \\
\texttt{laplace} & $65{,}536$ & 96 & cap, $1.18$--$1.19$ (from a peak of $1.25$) & 6/6 & \textbf{6/6} \\
\texttt{fem}     & $32{,}768$ & 96 & tolerance $10^{-6}$, within $9$ sweeps & 6/6 & \textbf{6/6} \\
\bottomrule
\end{tabular}
\caption{Paired consistent-vs-inconsistent runs across the three test
problems, power-weighted ($\ell=2$), $T=96$, with six matched seeds.
Entries count non-divergent runs: those that reached the tolerance
$10^{-6}$ (\texttt{fem}) or the $200$-sweep work cap with a finite
residual (\texttt{poisson}, \texttt{laplace}); the remaining runs diverged
to a non-finite residual. The fourth column gives the end state of the
non-divergent runs (where both modes are non-divergent, the entry
describes both; on \texttt{laplace} at $4n$ the two modes' residuals at
the cap differ by less than $0.01$), with the relative residual at the
cap, which was decreasing in every case. A finite residual at the cap
establishes neither eventual convergence nor asymptotic stability.}
\label{tab:consistency-sweep}
\end{table}

\subsubsection*{Step size, exponent, gather and warm-up}

Table~\ref{tab:beta} tests the two remedies suggested by the theory for
the consistent-reads instability, fixed under-relaxation and the adaptive
rule motivated by Theorem~\ref{thm:async-power-weighted-step}, on the six
matched seeds at
$T = 96$, with the thread-count sweep in its lower part. Fixed steps make
consistent reads non-divergent on every problem once $\beta$ is small
enough: from
$\beta = 0.5$ on \texttt{laplace} and \texttt{fem}, and from $\beta = 0.05$
on \texttt{poisson} ($3/6$ at $0.1$). Let $\beta^\ast$ denote the largest tested step at which every run was
non-divergent; with three or six runs per cell it brackets rather than
pins the threshold. On \texttt{poisson}, $\beta^\ast = 0.25$, $0.1$,
$0.1$, $0.1$ and $0.05$ at $T = 8$, $16$, $32$, $64$ and $96$, so
$T\beta^\ast$ stays between $1.6$ and $6.4$: the threshold falls with $T$
roughly as the prescription does, $3$--$13\times$ above $1/(2T)$ (or
$6$--$32\times$ if the largest step with any surviving run is used
instead), as expected of a bound for the worst case in which all $T$
threads collide; in the pile-on picture above, the steps that survive at
$T = 96$ correspond to $2/\beta \approx 20$--$40$ colliding threads. The price is speed:
the sweeps to $\|r\|_2/\|b\|_2 \le 0.05$ on \texttt{poisson} grow as
$1/\beta$ ($6.5$, $13$, $33$, $65$ and $130$ for $\beta = 0.1$ down to
$0.005$, against $1.2$ for inconsistent reads at $\beta = 1$), and on
\texttt{fem} the tolerance is reached after $31$, $73$ and $194$ sweeps at
$\beta = 0.5$, $0.25$ and $0.1$, against $8.9$, and not within the cap
below. The adaptive rule, evaluated per update from the IPR of the
gathered buffer with $\bar\tau = T - 1$, starts at $\beta \approx 1/(2T)$ on
the point-source transient and rises to $0.17$ on \texttt{poisson}
($\bar\beta = 0.14$ over the run); on \texttt{laplace} it stays at
$0.27$--$0.29$ throughout, and on \texttt{fem} it rises from $0.21$ to
$0.28$. It keeps consistent reads non-divergent in $6/6$ runs on
all three problems, although its late-run step exceeds the fixed-step failure threshold on \texttt{poisson}, because the damping is in place when
the pile-on occurs; it reaches $0.05$ in $11$ sweeps on \texttt{poisson},
ends the $200$ sweeps with the smallest residual of any consistent-reads
run there, and reaches the tolerance in $66$ sweeps on \texttt{fem}. Under
inconsistent reads every $\beta \le 0.5$ is non-divergent in $6/6$ runs and
the rule's results are indistinguishable from those under consistent
reads: once the step is this small the two read modes gave similar
results, and
with the adaptive rule the safety is bought at a $7$--$9\times$ cost in
sweeps relative to inconsistent reads at $\beta = 1$ (with the largest
fixed step at which every run was non-divergent, $11\times$ on
\texttt{poisson} and $3.5\times$ on \texttt{fem}). The $\beta = 1$ row is a fresh set of
runs of the configuration of Table~\ref{tab:consistency-sweep}, and
\texttt{poisson} under inconsistent reads diverged in $3/6$ of them against
$1/6$ there: the seed non-reproducibility noted above.

The third block of Table~\ref{tab:beta} varies the exponent under both
read modes at $\beta = 1$. Under inconsistent reads every run with $\ell \in \{0.5, 1, 4\}$ was
non-divergent on all three problems, while the repeated $\ell = 2$
configuration on \texttt{poisson} was non-divergent in $3/6$ (the marginal
case discussed above); on \texttt{fem} the sweeps to tolerance decrease
with $\ell$ ($11.0$, $9.7$, $8.9$ and $8.7$ for $\ell = 0.5$, $1$, $2$, $4$),
consistent with the sharper concentration. Under consistent reads the 
smaller exponents, whose sampling distributions are less concentrated 
and therefore collide less often, are the more robust: $\ell \in \{0.5, 1\}$ survive on
\texttt{laplace} and \texttt{fem} ($6/6$) where $\ell \in \{2, 4\}$ diverge
($0/6$). On \texttt{poisson} every exponent diverged under consistent
reads, consistent with every rule starting from the same sampling
distribution: the point-source residual has a single nonzero component,
on which every rule $\propto |r_j|^\ell$ with $\ell > 0$ (and its
$\ell \to 0^+$ limit on the support) places all of its mass, so the
iteration-$0$ pile-on begins identically for every exponent. Finally, filling the buffer
by the componentwise atomic gather of the inconsistent mode while still
committing the buffered value diverged in $18/18$ runs at baseline $n$
($\ell = 2$, $\beta = 1$, $T = 96$), as the bulk copy does, which supports
attributing the contrast between the two modes to the value reuse rather
than to the gather.

A natural-seeming mitigation, preceding power-weighted sampling by $K$
uniform warm-up sweeps so that the residual diffuses first, did not help.
A scan of $K \in \{0, 5, 10, 20, 50, 100, 150\}$ on \texttt{poisson} at
$T = 96$ with the six matched seeds gave consistent-reads divergence in
$42/42$ runs, and the same scan under inconsistent reads diverged in $14$
of the $36$ runs with $K > 0$ ($10$ to a non-finite residual and $4$ to a
finite residual above $10^{47}$ at the cap), a rate comparable to the $8$
of $25$ at $K = 0$ reported above. The outcome is consistent with
Theorem~\ref{thm:async-power-weighted} rather than a surprise. Uniform
sampling diffuses the point-source residual slowly, so at the switch the
IPR is still between $2\times10^3$ and $1.1\times10^4$ after $K = 5$
sweeps and $1.2$--$2.1\times10^2$ after $K = 150$ (last logged value
before the switch, across the six seeds), for which the certified delay
window, $\sqrt n/(2\nu_{\max}\rho_2)$, is at most about $5$ commits against a realized mean lag
of about $90$. The online IPR would thus have flagged every one of these
switches as outside the certified window. Once the pile-on starts it is
self-reinforcing, because $T$ concurrent commits to a high-residual
component create a new high-residual component that attracts the next
round of samples.

\begin{table}[ht]
\centering
\footnotesize
\setlength{\tabcolsep}{3pt}
\renewcommand{\arraystretch}{1.05}
\begin{tabular}{l cc cc cc}
\toprule
 & \multicolumn{2}{c}{\texttt{poisson}} & \multicolumn{2}{c}{\texttt{laplace}} & \multicolumn{2}{c}{\texttt{fem}} \\
$\beta$ & cons. & incons. & cons. & incons. & cons. & incons. \\
\midrule
1        & 0/6 & 3/6 (0.013) & 0/6 & 6/6 & 0/6 & 6/6 (8.9) \\
0.5      & 0/6 & 6/6 (0.014) & 6/6 & 6/6 & 6/6 (31.4) & 6/6 (32.0) \\
0.25     & 0/6 & 6/6 (0.017) & 6/6 & 6/6 & 6/6 (72.5) & 6/6 (72.9) \\
0.1      & 3/6 (0.021) & 6/6 (0.021) & 6/6 & 6/6 & 6/6 (194) & 6/6 (194) \\
0.05     & 6/6 (0.025) & 6/6 (0.025) & 6/6 & 6/6 & 6/6 [$4.7\!\times\!10^{-4}$] & 6/6 [$4.7\!\times\!10^{-4}$] \\
0.02     & 6/6 (0.032) & 6/6 (0.032) & 6/6 & 6/6 & 6/6 [$2.6\!\times\!10^{-2}$] & 6/6 [$2.6\!\times\!10^{-2}$] \\
0.01     & 6/6 (0.038) & 6/6 (0.038) & 6/6 & 6/6 & 6/6 [$0.12$] & 6/6 [$0.12$] \\
0.005    & 6/6 (0.045) & 6/6 (0.045) & 6/6 & 6/6 & 6/6 [$0.27$] & 6/6 [$0.27$] \\
adaptive & 6/6 (0.020) & 6/6 (0.020) & 6/6 & 6/6 & 6/6 (65.8) & 6/6 (66.0) \\
\midrule
\multicolumn{7}{l}{\textit{\texttt{poisson}, non-divergent runs by thread count}} \\
\multicolumn{7}{l}{\begin{tabular}{@{}r c cc ccccccc c@{}}
$T$ & $1/(2T)$ & $\beta\!=\!1$ incons. & $\beta\!=\!1$ cons. & $0.5$ & $0.25$ & $0.1$ & $0.05$ & $0.02$ & $0.01$ & $0.005$ & adaptive \\
4  & 0.125  & --  & --  & 3/3 & 3/3 & 3/3 & 3/3 & 3/3 & 3/3 & 3/3 & 3/3 \\
8  & 0.0625 & 6/6 & 0/6 & 1/3 & 3/3 & 3/3 & 3/3 & 3/3 & 3/3 & 3/3 & 3/3 \\
16 & 0.0313 & 6/6 & 0/6 & 0/3 & 2/3 & 3/3 & 3/3 & 3/3 & 3/3 & 3/3 & 3/3 \\
32 & 0.0156 & 6/6 & 0/6 & 0/3 & 0/3 & 3/3 & 3/3 & 3/3 & 3/3 & 3/3 & 3/3 \\
64 & 0.0078 & 6/6 & 0/6 & 0/3 & 2/3 & 3/3 & 3/3 & 3/3 & 3/3 & 3/3 & 3/3 \\
96 & 0.0052 & 5/6 & 0/6 & 0/3 & 0/3 & 0/3 & 3/3 & 3/3 & 3/3 & 3/3 & 3/3 \\
\end{tabular}} \\
\midrule
\multicolumn{7}{l}{\textit{exponent at $\beta = 1$, $T = 96$, six matched seeds}} \\
$\ell = 0.5$ & 0/6 & 6/6 (0.015) & 6/6 & 6/6 & 6/6 (11.2) & 6/6 (11.0) \\
$\ell = 1$   & 0/6 & 6/6 (0.014) & 6/6 & 6/6 & 6/6 (9.8)  & 6/6 (9.7) \\
$\ell = 2$   & 0/6 & 3/6 (0.013) & 0/6 & 6/6 & 0/6        & 6/6 (8.9) \\
$\ell = 4$   & 0/6 & 6/6 (0.012) & 0/6 & 6/6 & 0/6        & 6/6 (8.7) \\
\bottomrule
\end{tabular}
\caption{Step-size control at $T = 96$ ($\ell = 2$, six matched seeds,
$200$-sweep cap). Each cell counts non-divergent runs and gives, in
parentheses, the median relative residual at the cap (\texttt{poisson};
on \texttt{laplace} it is $0.99$--$1.2$ in every non-divergent run) or the
median number of sweeps to the tolerance $10^{-6}$ (\texttt{fem}; in
brackets, the residual at the cap when the tolerance was not reached).
``adaptive'' is the rule
$\beta_j = 1/(1 + 2\nu(\tilde r^{(j)})\rho_2(T-1)/\sqrt n)$ evaluated per
update on the gathered buffer, with $\rho_2 = \sqrt{5/4}$ on
\texttt{laplace} and \texttt{poisson} and the interior-row value
$\sqrt{9/8}$ on \texttt{fem} (the two boundary-adjacent rows have
$\sqrt{19/16}$, a difference of about $2\%$ in $\beta_j$). Lower part: \texttt{poisson} under
consistent reads, non-divergent runs of three seeds at each thread count
($50$-sweep cap; the $T = 96$ entries are the first three seeds of the
upper part, which are the three that diverged at $\beta = 0.1$ there), with
the $\beta = 1$ columns of the original six-seed sweep
under both read modes for reference. The consistent-reads mode at
$\beta = 1$ diverged at every tested thread count (see the pile-on model
in the text). Third block: the exponent under both read modes at $\beta = 1$; the
$\ell = 2$ row is the $\beta = 1$ row of the first block.}
\label{tab:beta}
\end{table}

The experimental panels suggest the IPR
controls both how much power-weighted sampling helps in expectation and
how much it hurts under concurrency. On problems where $\nu^2$ stays
bounded as $n$ grows (\texttt{laplace}, \texttt{fem}), both the progress
amplification and the collision rate scale modestly, and quadrupling $n$
is enough to make consistent reads non-divergent
(Table~\ref{tab:consistency-sweep}).
On problems where $\nu^2$ scales with $n$ (\texttt{poisson} with its
point-source forcing), the progress amplification grows linearly with
problem size but the collision rate stays roughly constant, so consistent reads diverged at every $n$ we tested. In both regimes,
inconsistent reads appears to provide the sign-aware feedback needed to
absorb most collisions. 

\section{Discussion and Conclusion}
\label{sec:conclusion}

For residual-weighted randomized Jacobi, the standard guaranteed progress
coefficient can be improved by a factor equal to the inverse participation
ratio of the residual.
Theorem~\ref{thm:sharp-power} makes this precise for $\ell = 2$: the
guaranteed per-step progress coefficient is multiplied by $\nu^2$ (and the
expected progress itself, for constant-diagonal $A$), a scalar that is at
least $1$, exceeds $1$ whenever the residual is non-uniform, and can be
accumulated in the sampling pass. In the experiments the IPR is
sustained in the range $3.5$--$6$ on all three test problems
and between $3$ and $16$ on eight Matrix Market matrices, with much
larger values during the early phase, where most of the residual
reduction happens, on the problems whose initial residual is localized.

The asynchronous extension carries this analysis into the ADG framework
under a global IPR envelope, at a price: the Cauchy--Schwarz step
costs a factor $\sqrt n$ in the cross term, so
Theorem~\ref{thm:async-power-weighted} and ADG's result are complementary
bounds for different sampling rules rather than one subsuming the other. In
the hardware runs the IPR trajectory is nearly independent of thread count
across two decades of concurrency, and the iteration remains
non-divergent, in most runs, far beyond the delay window the theorem
certifies, so the $\sqrt n$ looseness does not appear binding for the
problem classes considered even though the theoretical gap remains.

The least anticipated finding is that the ADG ordering of the read models
is reversed on our test problems
(Section~\ref{sec:experiments:results}): consistent reads, the easy case
for the theory, is the mode that diverged under power-weighted sampling at
high concurrency, and the proposed mechanism ties this to the same
$\nu^2$, which enters the per-pair collision rate as it enters the
progress bound. The problem-size, step-size, exponent and common-gather
experiments (Tables~\ref{tab:consistency-sweep} and~\ref{tab:beta}) are
consistent with the proposed mechanism, and uniform warm-up sweeps before
the switch to power weighting did not prevent the instability. 
The same ordering of the read models, and its repair by under-relaxation,
were also observed for asynchronous randomized Kaczmarz with
residual-adaptive row selection~\cite{coleman2026adaptive}, which suggests
that the effect is not specific to the Jacobi update.

The
adaptive step-size rule motivated by
Theorem~\ref{thm:async-power-weighted-step} removed the instability in
every run on every problem, at a $7$--$9\times$ cost in sweeps relative to
inconsistent reads at full step, and fixed under-relaxation did so once
the step was small enough.
For problems like these, the practical recommendation is to refresh the
selected residual component before each update, or to under-relax,
whenever power-weighted sampling is used at non-trivial thread count:
refreshing reduced but did not eliminate divergence, and the damped step
removed it in the tested configurations. The IPR is the quantity that
governs both the algorithm's productivity and its observed sensitivity to
concurrency.

Three directions follow. The most pressing is an inconsistent-reads analog
of Theorem~\ref{thm:async-power-weighted} that captures the feedback
damping, which requires tracking the sign correlation between the stale
read and the in-flight update rather than bounding the cross term in
absolute value; since refreshing the component at update time turns out to
be preferable to reusing the buffered value for the problem class of
interest, this is the direction that most directly informs practice. The
second is to replace the global IPR envelope with a martingale or
stopping-time argument that tracks $\nu^2(\tilde r^{(j)})$ directly, and,
for specific matrix classes such as discretized elliptic problems with
localized forcing, to prove $\nu_i^2 \ge \bar\nu^2 > 1$ along the
trajectory, which would make Corollary~\ref{cor:global-rate} unconditional
where its margin is largest. The third is the extension of the sharpened and asynchronous 
bounds to general $\ell$, the choice of $\ell$ as a function of residual concentration and
concurrency, and the connection of residual-weighted sketch selection to
the sketch-and-project framework.



\section*{Declarations}

\bmhead{Funding}
This work leveraged the ACES Cluster at Texas A\&M University under
allocation CIS250436 from the Advanced Cyberinfrastructure Coordination
Ecosystem: Services \& Support (ACCESS) program, which is supported by
U.S. National Science Foundation grants \#2138259, \#2138286, \#2138307,
\#2137603, and \#2138296.

\bmhead{Competing interests}
The author declares no competing interests.

\bmhead{Data and code availability}
The solver, SLURM run scripts, and data/logs behind every table and figure are available from the author
upon reasonable request.

\bmhead{Ethics approval and consent for publication}
Not applicable.

\bmhead{Author contributions}
The sole author performed the research and prepared the manuscript.

\bmhead{Use of AI tools}
During the preparation of this work the author used Claude (Anthropic) to
assist in: checking and revising proofs, editing the text,
and in writing analysis/plotting code. The author reviewed and verified
the content and takes full responsibility for it.

\begin{appendices}

\section{Detailed Proof of Theorem~\ref{thm:async-power-weighted-step}}
\label{sect:detailed-proofs}
\begin{proof}
\label{proof:detailed-step-size}
\noindent\textit{Step 1: $\beta$-modified one-step identity.} With
$\beta \in (0, 1]$, the update gives
$e^{(j+1)} = e^{(j)} + \beta\,\tilde r_{k_j}^{(j)} e_{k_j}$. Each
committed update in the in-flight residual decomposition now carries an
additional factor of $\beta$:
\[
r^{(j)} = \tilde r^{(j)} - \beta \sum_{l \in J_j} \tilde r_{k_l}^{(l)} A e_{k_l}.
\]
Expanding $\Anorm{e^{(j+1)}}^2$ with these factors,
\begin{equation}
\label{eq:async-onestep-beta}
\Anorm{e^{(j+1)}}^2 = \Anorm{e^{(j)}}^2 - \beta(2-\beta)\,(\tilde r_{k_j}^{(j)})^2 + 2\beta^2 \sum_{l \in J_j} A_{k_j, k_l}\,\tilde r_{k_j}^{(j)}\,\tilde r_{k_l}^{(l)}.
\end{equation}
Compared to the $\beta = 1$ identity~\eqref{eq:async-onestep-proof}, the
progress coefficient scales as $\beta(2-\beta)$ (from the $-2\beta + \beta^2$
combination in the expansion) and the interference scales as $\beta^2$
(one factor of $\beta$ from the current update and one from the in-flight
update being measured against it).

\smallskip\noindent\textit{Step 2: Conditional expectations and
telescoping.} Steps~2--4 of the proof of
Theorem~\ref{thm:async-power-weighted} apply with the same conditional
bounds, modulated only by the $\beta$ scalings
from~\eqref{eq:async-onestep-beta}. The telescoped bound becomes
\[
E_m \;\leq\; E_0 \;-\; \frac{\nu_{\min}^2\,\lmin}{n}\,\tilde\nu_\tau(\beta)\,\sum_{j=0}^{m-1} \E[\Anorm{\tilde e^{(j)}}^2],
\]
where the effective convergence parameter combines the progress and
interference $\beta$-scalings:
\[
\tilde\nu_\tau(\beta) \;:=\; \beta(2-\beta) - 2\beta^2\rho_{\text{dyn}}\tau \;=\; 2\beta - \beta^2(1 + 2\rho_{\text{dyn}}\tau).
\]
Convergence requires $\tilde\nu_\tau(\beta) > 0$, which is equivalent to
$\beta < 2/(1 + 2\rho_{\text{dyn}}\tau)$. Unlike the $\beta = 1$ case,
no upper bound on $\rho_{\text{dyn}}\tau$ is required: for any finite
delay and any finite IPR bound, the convergence window for $\beta$ is
non-empty.

\smallskip\noindent\textit{Step 3: Optimal step size.} Maximizing
$\tilde\nu_\tau(\beta)$ over $\beta \in (0, 2/(1 + 2\rho_{\text{dyn}}\tau))$:
\[
\frac{d}{d\beta}\tilde\nu_\tau(\beta) = 2 - 2\beta\,(1 + 2\rho_{\text{dyn}}\tau) = 0 \quad\Longrightarrow\quad \tilde\beta = \frac{1}{1 + 2\rho_{\text{dyn}}\tau}.
\]
The optimizer $\tilde\beta$ lies in $(0, 1]$ for every
$\rho_{\text{dyn}}\tau \geq 0$, so the choice is always feasible (and
reduces to $\beta = 1$ in the synchronous limit $\tau = 0$). Substituting:
\[
\tilde\nu_\tau(\tilde\beta) = \frac{2}{1 + 2\rho_{\text{dyn}}\tau} - \frac{1 + 2\rho_{\text{dyn}}\tau}{(1 + 2\rho_{\text{dyn}}\tau)^2} = \frac{1}{1 + 2\rho_{\text{dyn}}\tau}.
\]

\smallskip\noindent\textit{Step 4: Epoch reduction.} The epoch-assembly
argument from Step~5 of the proof of
Theorem~\ref{thm:async-power-weighted} applies with the $\beta$ scalings
of~\eqref{eq:async-onestep-beta}: the lower bound on the error becomes
$E_j \geq E_0 - \omega(\beta)\,S_j$ with
\[
\omega(\beta) \;=\; \frac{\nu_{\max}^2\,\lmax}{n}\,\big(\beta(2-\beta) + 2\beta^2\rho_{\text{dyn}}\tau\big),
\]
so that $\sum_{j<m}\E[\Anorm{\tilde e^{(j)}}^2] \geq E_0/(2\omega(\beta))$
for $m \geq \ln(2)/\omega(\beta)$. At $\beta = \tilde\beta$ we have
$2\tilde\beta\rho_{\text{dyn}}\tau = 1 - \tilde\beta$, hence
$\omega(\tilde\beta) = \nu_{\max}^2\,\lmax\,\tilde\beta\,(3 - 2\tilde\beta)/n$.
Substituting this and $\tilde\nu_\tau(\tilde\beta) = \tilde\beta$ into the
telescoped bound of Step~2:
\[
E_m \;\leq\; \left(1 - \frac{\nu_{\min}^2\,\lmin}{n}\cdot\frac{\tilde\nu_\tau(\tilde\beta)}{2\,\omega(\tilde\beta)}\right) E_0 \;=\; \left(1 - \frac{\nu_{\min}^2}{\nu_{\max}^2}\cdot\frac{1}{2\kappa\,(3 - 2\tilde\beta)}\right) E_0.
\]
\end{proof}

\end{appendices}

\bibliography{references}

\end{document}